\documentclass[bj,preprint]{imsart}

\RequirePackage{amsthm,amsmath,amsfonts,amssymb}
\RequirePackage[numbers]{natbib}

\RequirePackage[colorlinks,citecolor=blue,urlcolor=blue]{hyperref}

\startlocaldefs
\usepackage{mathtools}
\usepackage{bbm}
\usepackage{amsthm, amsmath, amssymb, amscd, latexsym}
\usepackage{color}

\usepackage[utf8]{inputenc}

\newtheorem{theorem}{Theorem}[section]
\newtheorem{lemma}[theorem]{Lemma}
\newtheorem{proposition}[theorem]{Proposition}
\newtheorem{definition}[theorem]{Definition}

\newtheorem{corollary}[theorem]{Corollary}

\newtheorem{remark}[theorem]{Remark}

\newcommand{\mods}[1]{\textcolor{black}{#1}\index {#1}}

\def \N {{\mathbb N}}
\def \R {{\mathbb R}}

\newcommand{\norm}[1]{\left\|#1 \right\|}

\newcommand{\rp}{\mathbb{P}}

\DeclareMathOperator*{\argmax}{arg\,max}

\def \sign{{\rm sign}}

\newcommand{\var}{{\rm Var}}

\newcommand{\gep}{\epsilon}

\newcommand{\E}{\mathbb{E}}
\newcommand{\uf}{u}

\newcommand{\e}{{\rm e}}

\newcommand{\F}{\mathcal{F}}

\newcommand{\X}{\mathcal{X}}

\definecolor{darkred}{rgb}{.7,0,0}

\definecolor{darkgreen}{rgb}{0,0.5,0}

\definecolor{darkblue}{rgb}{0,0,0.7}

\newcommand{\rkhs}{H}

\newcommand\pe{\mu}
\newcommand\sh{\mathcal{Q}}
\newcommand\om{\mathcal{Z}}

\newcommand\f{\bar{f}}
\newcommand\g{\bar{g}}
\newcommand\gammal{\bar\gamma}
\newcommand\varphil{\bar\varphi}

\newcommand\sha{\mathcal{Q}_\alpha}
\newcommand\oma{\mathcal{Z}_\alpha}
\newcommand\pel{\bar{\mu}}
\newcommand\shl{\bar{\mathcal{Q}}_\alpha}
\newcommand\oml{\bar{\mathcal{Z}}_\alpha}

\newcommand\ms{T}
\newcommand\mss{\mathcal{T}}
\newcommand\dens{\pi}

\endlocaldefs

\endlocaldefs

\begin{document}

\begin{frontmatter}
\title{Rates of contraction of posterior distributions based on $\lowercase{p}$-exponential priors}
\runtitle{Posterior contraction for $\lowercase{p}$-exponential  priors}

\begin{aug}

\author[a]{\fnms{Sergios} \snm{Agapiou}\ead[label=e1]{agapiou.sergios@ucy.ac.cy}},
\author[b]{\fnms{Masoumeh} \snm{Dashti}\ead[label=e2]{m.dashti@sussex.ac.uk}}
\and
\author[c]{\fnms{Tapio} \snm{Helin}\ead[label=e3]{tapio.helin@lut.fi}}

\address[a]{Department of Mathematics and Statistics, University of Cyprus.\\ \printead{e1}}
\address[b]{Department of Mathematics, University of Sussex.\\\printead{e2}}
\address[c]{School of Engineering Science, LUT University.\\ \printead{e3}}

\end{aug}
\runauthor{S. Agapiou, M. Dashti and T. Helin}

\begin{abstract}
We consider a family of infinite dimensional product measures with tails between Gaussian and exponential, which we call $p$-exponential measures. We study their measure-theoretic properties and in particular their concentration. Our findings are used to
develop a general contraction theory of posterior distributions on nonparametric models with $p$-exponential priors in separable Banach parameter spaces. 
Our approach builds on the general contraction theory for Gaussian process priors in \cite{VZ08}, namely we use prior concentration to verify prior mass and entropy conditions sufficient for posterior contraction. 
However, the specific concentration properties of $p$-exponential priors lead to a more complex entropy bound which can influence negatively the obtained rate of contraction, depending on the topology of the parameter space. Subject to the more complex entropy bound, we show that the rate of contraction depends on the position of the true parameter relative to a certain Banach space associated to $p$-exponential measures and on the small ball probabilities of these measures. \mods{For example, we apply our theory in the white noise model under Besov regularity of the truth and obtain minimax rates of contraction using (rescaled) $\alpha$-regular $p$-exponential priors. In particular, our results suggest that when interested in spatially inhomogeneous unknown functions, in terms of posterior contraction, it is preferable to use Laplace rather than Gaussian priors.} \end{abstract}


\end{frontmatter}

\section{Introduction}\label{sec:intro}
Gaussian processes are routinely used as priors in many nonparametric inference problems, for example in spline smoothing  \cite{KW70}, density estimation \cite{PL88}, nonparametric regression \cite{SKW99}, inverse problems \cite{stuart} and drift estimation of diffusions \cite{PPRS12}. At the same time, there is a growing number of problems 
for which it is preferable to utilize heavier-tailed priors, while maintaining the favourable convexity properties offered by the Gaussian distribution. 
For example, priors constructed using infinite products of Laplace distributions are extensively used in the literature of Bayesian
inverse problems in the form of Besov-space priors with integrability parameter $p = 1$  \cite{LSS09, DS15, KLNS12, HB15, ABDH18}.
Such priors on the one hand have attractive sparsity-promoting properties at the level of maximum a posteriori estimates \cite{KLNS12,ABDH18}, and on the other hand are logarithmically concave, thus computationally and analytically tractable. Besov-space priors are defined through expansions in a wavelet basis and for $p=1$ use  $\ell_1$-type penalization on the corresponding coefficients, an idea widely-used in the statistical literature  \cite{rev5, rev6, rev7, rev8,rev9,rev10}.


The study of the asymptotic performance of posterior distributions in the infinitely-informative data-limit, under the frequentist assumption that the available data is generated from an underlying fixed value of the unknown, has received great attention in the last two decades. In particular, there has been enormous progress in the study of rates of posterior contraction, that is the concentration rates of posterior distributions around the underlying value of the unknown. The works of Ghosal and van der Vaart \cite{GGV00} and Shen and Wasserman \cite{SW01} for independent and identically distributed (i.i.d.) observations, together with the work of Ghosal and van der Vaart \cite{GV07} for non-i.i.d. observations, paved the way for a comprehensive theory for rates of posterior contraction under general assumptions on the prior and model. 

For Gaussian priors, posterior contraction has been vigorously studied aided by the available very deep understanding of Gaussian processes; see for example \cite{VZrkhs} for a presentation of the relevant elements of Gaussian process theory. Of great importance in this context, has been the work of van der Vaart and van Zanten \cite{VZ08}, who studied general posterior contraction based on the concentration properties of the Gaussian prior. In particular, they showed that the rate of contraction depends on the position of the true parameter underlying the data relative to the reproducing kernel Hilbert space and the centered small ball probabilities of the Gaussian prior. An incomplete list of other contributions which advanced the theory of posterior contraction  under Gaussian priors in several models, often using mixtures of Gaussian processes to achieve adaptation, includes \cite{BG03, IC08, GN11, KVZ12, PSZ13, SVZ13, VZ07, VZ09,  LZ00}. See also the recent books \cite{GV17, GN16}.

On the contrary, the frequentist asymptotic performance of posterior distributions arising from infinite-dimensional Laplace-type priors is much less understood. In particular, there is no general theory for posterior contraction and the only applicable contraction result we are aware of, refers to undersmoothing product priors in the white noise model \cite[Corollary 3]{CN13}. Of some relevance are existing posterior contraction results under sieve priors, which include randomly truncated products of exponential distributions \cite{AGR13, KR13}. For such priors a mechanism for choosing the truncation point is necessary either using a hyperprior or with an empirical Bayes procedure.  \mods{Also relevant, is the result of \cite[Section 3]{CSV15} for the independent and identically distributed product Laplace prior in the sparse Gaussian sequence model setting.}

In this work, we consider a class of infinite-dimensional priors spanning between Gaussian and Laplace product priors. We call such priors \emph{$p$-exponential}, with $p\in[1,2]$ reflecting the tail behaviour, where $p=2$ corresponds to Gaussian and $p=1$ to exponential tails. {Our aim is twofold: first, to develop the relevant measure theory for these priors and to study their concentration properties and second, to study posterior contraction for general models based on prior concentration, analogously to the Gaussian contraction theory in \cite{VZ08}. }


\subsection{General posterior contraction theory}
Consider the problem of inferring an unknown parameter $\theta\in\Theta$ from observations $X^{(n)}$ drawn from distributions $P^{(n)}_\theta$, where $n\to\infty$ corresponds to the infinitely-informative data-limit. We put a prior $\Pi$ on $\theta$ and aim to study the frequentist asymptotic properties of the resulting posterior distribution on $\theta$ after observing $X^{(n)}$, $\Pi_n(\cdot|X^{(n)})$. In particular, we make the frequentist assumption that the available observations have been generated from a fixed underlying true parameter $\theta_0\in\Theta$, and we are interested in investigating the concentration rate of the posterior distribution around the truth in the limit $n\to\infty$. We say that the posterior distribution contracts with a rate $\epsilon_n$ at $\theta_0$ with respect to a metric $d$ on $\Theta$, if $\Pi_n(\theta: d(\theta,\theta_0)\geq M_n\epsilon_n|X^{(n)})\to0$ in $P^{(n)}_{\theta_0}$-probability, for every $M_n\to\infty$.

Posterior contraction in this general-prior and general-model setup, has been studied by Ghosal and van der Vaart in \cite{GV07}. 
Given a model and distance $d$, assuming that there exist exponentially powerful tests for separating $\theta_0$ from $d$-balls at a certain distance from it, 
they derived conditions on the prior securing that an $\epsilon_n$ is a rate of contraction around $\theta_0$ with respect to $d$: the prior needs to put sufficient mass around the true $\theta_0$ and almost all its mass on sets of bounded complexity. These conditions are expressed via norms and discrepancies which are relevant to the statistical setting of interest. In particular, they involve both neighbourhoods of $\theta_0$ expressed via the metric $d$, as well as neighbourhoods of $P^{(n)}_{\theta_0}$ expressed via Kullback-Leibler divergence and variations. For a comprehensive and up to date treatment see \cite[Chapter 8]{GV17}.

\subsection{Gaussian concentration and posterior contraction}\label{ssec:gauss}
We briefly describe the posterior contraction theory for Gaussian priors of van der Vaart and van Zanten \cite{VZ08}, which relies on a good understanding of the concentration properties of Gaussian measures; see also \cite[Chapter 11]{GV17}.

Let  $(X, \norm{\cdot})$ be a  separable Banach space and let $\mu$ be a centered Gaussian prior in $X$. Denote by $\rkhs$ the reproducing kernel Hilbert space (RKHS) of $\mu$, with corresponding norm $\norm{\cdot}_\rkhs$. Moreover, denote by $B_X$ the closed unit ball of $X$ centered at the origin. The concentration properties of $\mu$ at a point $w$ in the topological support of $\mu$, ${\rm supp}(\mu)=\overline{H}^{\norm{\cdot}_X}$, were shown to be captured by the concentration function 
\begin{equation}\label{eq:gaussian_conc}
\varphi_w(\epsilon)=\inf_{h\in\rkhs:\norm{h-w}\leq \epsilon}\frac12 \norm{h}^2_H-\log\mu(\epsilon B_X), \;\epsilon>0.
\end{equation}For $w=0$, the first term vanishes and the concentration function measures the probability of centered balls of size $\epsilon$ in $X$. The idea is that for nonzero $w\in {\rm supp}(\mu)$, the concentration function measures the probability of balls of radius $\epsilon$ centered at $w$, with the first term measuring the loss of probability due to shifting
from centered to noncentered balls; this is made precise by the bounds in \cite[Lemma 5.3]{VZrkhs}. 

Using the above interpretation of the concentration function, together with a concentration inequality due to Borell, \cite[Theorem 3.1]{CB75}, van der Vaart and van Zanten showed in \cite[Theorem 2.1]{VZ08} that for a $w_0\in {\rm supp}(\mu)$, if $\epsilon_n$ satisfies
\begin{equation}\label{eq:gcon}
\varphi_{w_0}(\epsilon_n)\leq n\epsilon_n^2,
\end{equation} then the prior puts a certain minimum mass in $\epsilon_n$ balls in $X$ around the $w_0$  and it is possible to find $\Theta_n\subset X$ which contains the bulk of the prior mass and has exponentially bounded complexity. 
These assertions point to the conditions of general-model general-prior results discussed in the previous subsection, see for example \cite[Theorem 8.9 and Theorem 8.19]{GV17}. However unlike the conditions of these general results which involve statistically relevant norms and discrepancies, the  assertions of \cite[Theorem 2.1]{VZ08}  are expressed purely in the Banach space norm. To bridge this gap and indeed prove that $\epsilon_n$ is a posterior contraction rate in specific statistical settings, one needs to relate the statistically relevant quantities appearing in general-model general-prior results to the Banach space norm.

In a range of models, this reconciliatory work has been done in \cite{GV07} in the general-prior context, and there exist general-prior contraction theorems with assumptions purely expressed in the Banach space norm \cite{GV07}; for example see \cite[Theorem 8.31]{GV17} in the white noise model, or \cite[Theorem 8.26]{GV17} in the normal fixed-design regression setting. 
In other models such as density estimation or nonparametric binary classification, the reconciliatory work has been done in the context of Gaussian priors in \cite{VZ08}, see \cite[Lemmas 3.1 and 3.2]{VZ08} respectively, and no general-prior theorems were explicitly formulated. 
We stress that the reconciliatory work for these models is not explicit to Gaussian priors, thus the proofs of all the Gaussian contraction results found in \cite[Section 3]{VZ08}, can be easily used to get contraction results for priors for which analogous results to \cite[Theorem 2.1]{VZ08} hold. 

\subsection{Our contribution}
In the present paper we consider parameter spaces $X$ which are separable Banach and which possess a Schauder basis. We use the Schauder basis to construct $p$-exponential measures in $X$, by identifying them to infinite products of independent univariate $p$-exponential distributions. Our main contribution is that we generalize the aforementioned Gaussian general contraction theorem \cite[Theorem 2.1]{VZ08} to $p$-exponential measures, and to achieve this we develop the necessary concentration theory for $p$-exponential measures. The obtained general contraction result enables the study of contraction rates of posterior distributions based on $p$-exponential priors, in a range of standard nonparametric statistical models. A brief summary of the paper is as follows.

In Section \ref{sec:props}, we introduce $p$-exponential measures in $X$ 
and study their properties relating to convexity, equivalence and singularity under translations, topological support and ultimately concentration. 
We find that the concentration of a $p$-exponential measure at a point $w$ in its support, depends on the position of $w$ relative to a Banach space, rather than relative to a Hilbert space as was the case for Gaussian measures. We define the corresponding concentration function $\varphi_w(\cdot)$ and show in Theorem \ref{thm:cf} that it has a similar interpretation to the Gaussian concentration function. In Proposition \ref{p:tal}, we  derive a concentration inequality for $p$-exponential measures, which follows from Talagrand's work in \cite{TA94} and, although substantially more intricate, is analogous to the aforementioned Gaussian concentration inequality \cite[Theorem 3.1]{CB75} used for studying contraction in \cite{VZ08}. 

In Section \ref{sec:main} we use the interpretation of the concentration function, together with the available concentration inequality to generalize the Gaussian contraction result \cite[Theorem 2.1]{VZ08} to $p$-exponential measures in Theorem \ref{t:main}, which is the main result of this paper. 
Since the concentration properties of $p$-exponential measures are more intricate, we get a more complicated complexity bound compared to the Gaussian case. 

In Section \ref{sec:mod}, we present posterior contraction results for general $p$-exponential measures in two standard statistical models: the white noise model and density estimation. These results  follow immediately from Theorem \ref{t:main}, as discussed at the end of the last subsection. 

\mods{In Section \ref{sec:l2} we consider $\alpha$-regular $p$-exponential priors in separable Hilbert spaces, study bounds on the corresponding concentration function for Besov-type regularity of the truth and compute posterior contraction rates in the white noise model, with $L_2$ loss. In this case, the complexity bound in Theorem \ref{t:main}, which is more complicated for $p$-exponential priors compared to Gaussian priors, does not affect the rates. Our bounds are particularly interesting for Besov spaces of spatially inhomogeneous functions, that is for Besov integrability parameter $q<2$. In this case, Gaussian priors appear to be suboptimal, more specifically to be limited by the minimax rate over linear estimators, which is slower than the minimax rate \cite[Theorem 1]{DJ98}.  On the other hand, $p$-exponential priors with $p<2$ can do better than the linear minimax rate, see Theorem \ref{thm:wn2} and Remark \ref{rem:l2ratesqle2}. Furthermore, we can achieve the minimax rate using {rescaled} undersmoothing $p$-exponential priors for $p=q$, or the minimax rate up to lower order logarithmic terms using rescaled undersmoothing $p$-exponential priors with $p<q$; see Theorem \ref{thm:wn3} and Remark \ref{rem:rem3}. To our knowledge, this is the first occurrence of a prior achieving the minimax rate over such Besov spaces in the literature. Although we only show upper bounds on the rate of contraction, our results indicate that when interested in spatially inhomogeneous unknowns, in terms of posterior contraction rates, it is beneficial to use Laplace rather than Gaussian priors.} 

In Section \ref{sec:l8} we consider $\alpha$-regular $p$-exponential priors constructed via wavelet expansions in the space of continuous functions on the unit interval, $C[0,1]$, and study bounds on the corresponding concentration function in the supremum norm, under H\"older-type regularity of the truth. To this end, we prove new centered small ball probability bounds in the supremum norm for $p$-exponential measures, see Proposition \ref{prop:supnorm}. We then compute posterior contraction rates in density estimation, with Hellinger-distance loss. In this case, the rates are affected by the more complicated complexity bound in Theorem \ref{t:main} and appear to be suboptimal, see Theorem \ref{thm:decontr}.

\mods{The proofs of our results are contained in the supplementary material below, together with some complementary technical results and discussions.}

\subsection{Notation} We denote by $\R^\infty$ the space of all real sequences and by $\mathcal{B}(\R^\infty)$ the Borel $\sigma$-algebra with respect to the product topology. We denote by $\ell_p$ the space of $p$-summable real sequences. 
The space of square integrable real functions on the unit interval is denoted by $L_2[0,1]$, while $C[0,1]$ is the space of continuous real functions on the unit interval with the supremum norm. For $s>0$, we use $C^s=C^s[0,1]$ to denote the space of $s$-H\"older real functions on the unit interval . For a normed space $(Y,\norm{\cdot}_Y)$, we denote by $B_Y$ the closed unit ball in $Y$. The notation $N(\epsilon, A, d)$ is used for the $\epsilon$-\emph{covering number} of a subset $A$ of a metric space with metric $d$, that is the minimum number of balls of radius $\epsilon$ with respect to $d$ which are needed to cover the set $A$. For two positive sequences $(a_n), (b_n)$, $a_n\asymp b_n$ means $a_n/b_n$ is bounded away from zero and infinity, while $a_n\lesssim b_n$ means that $a_n/b_n$ is bounded.
\section{$p$-exponential measures and their properties}\label{sec:props}
In this section we introduce $p$-exponential measures and study some of their properties. In particular, we discuss their convexity, 
 behaviour under translations, topological support and  concentration properties. 

\subsection{$p$-exponential measures}\label{ssec:defn}
\begin{definition}\label{def:pexp}
Let $\gamma=(\gamma_\ell)_{\ell\in \N}$ be a deterministic decaying sequence of positive real numbers and let $\xi_\ell, \,\ell\in\N$, be independent and identically distributed real random variables with probability density function $f_p(x)\propto\exp(-\frac{|x|^p}p)$, $x\in\R$ for $p\in[1,2]$. We define the probability measure $\pe$ on the measurable space  $(\R^\infty, \mathcal{B}(\R^\infty))$ to be the law of the sequence $(\gamma_\ell \xi_\ell)_{\ell\in\N}$ and call it a \emph{$p$-exponential measure with scaling sequence $\gamma$}. 
\end{definition}

In the following we will often suppress the dependence on $\gamma$ and call $\pe$ a $p$-exponential measure. For $p=1$ and $p=2$ we get centered Laplace and centered  Gaussian measures respectively, both in sequence space. 
While we restrict $p$ between 1 and 2,
        many of the results in this section as for example the ones in the following subsection on convexity, clearly hold in greater generality and in particular for $p\geq1$. However, our treatment on the concentration of $p$-exponential measures in Subsection \ref{ssec:conc}, is explicit to $p\in[1,2]$.  
 
 Depending on the decay properties of $\gamma$, draws from $\pe$ almost surely belong to certain subspaces of $\R^\infty$. 
For example, $\gamma\in\ell_2$ if and only if $\pe(\ell_2)=1$, see \cite[Lemma S.M.1.2]{APSS17}. Lemma \ref{lem:bessup} in Section \ref{sec:l2} below studies Besov-type regularity of $\pe$, for certain choices of the scaling sequence $\gamma$; this result includes Sobolev-type regularity as a special case.

Any Gaussian random element in a separable Banach space can be identified with a Gaussian product measure as above with $p=2$, for example using the Karhunen-Loeve expansion  \cite[Theorem 2.6.10]{GN16}. Likewise, a $p$-exponential measure can be identified naturally with a measure on a separable Banach space $X$, provided $X$ possesses a Schauder basis, which can be normalized or not. 

\begin{definition}
\mods{Let $(X, \norm{\cdot}_X)$ be a separable Banach space. A Schauder basis is a sequence $\{\psi_\ell\}\subset X$, such that for every $u\in X$, there exists a unique real sequence $(u_\ell)_{\ell\in\N}$, so that \[u=\sum_{\ell=1}^\infty u_\ell\psi_\ell,\] where the convergence is with respect to $\norm{\cdot}_X$.}\end{definition}
For example, if $\gamma\in \ell_2$,  a $p$-exponential measure can be identified with a measure on a subspace of the space of square integrable functions on the unit interval, $X=L_2[0,1]$, via the random series expansion  \begin{equation}\label{eq:kl}u(x)=\sum_{\ell=1}^\infty \gamma_\ell \xi_\ell\psi_\ell(x),\end{equation}
where $\{\psi_\ell\}$ is an orthonormal basis in $L_2[0,1]$. It can also be identified with a measure on the space of continuous functions on the unit interval, $X=C[0,1]$, 
{using a similar random series expansion, where $\{\psi_\ell\}$ is a Schauder basis in $C[0,1]$; see Section \ref{sec:l8} below.}

In the general  separable Banach space setting, we also have that depending on the speed of decay of the scaling sequence $\gamma$, draws from a $p$-exponential measure almost surely belong to subspaces of $X$. If $X$ is a function-space, these subspaces correspond to a form of  higher regularity. We stress here, that such function-space regularity is not solely linked to the speed of decay of $\gamma$, but also depends on the scaling and regularity of the Schauder basis $\{\psi_\ell\}$. For example, one can study the H\"older regularity of draws using the Kolmogorov Continuity Test. See \cite[Corollary 5]{DS15} for a result under general conditions on the Schauder basis and scaling sequence, or Proposition \ref{prop:hol} in Section \ref{sec:l8} below for a result under more specific conditions.

While developing our posterior contraction theory for $p$-exponential priors below, we will use the sequence space or the general separable Banach space representation of the measure $\pe$ interchangeably. The particular random series expansion representation, and specifically the choice of the Schauder basis, will become relevant through the concentration function when actually computing the contraction rate in specific settings with specific priors in Sections \ref{sec:l2} and \ref{sec:l8}.

\subsection{Convexity}

We next study the convexity properties of $p$-exponential measures. The convexity of measures in infinite dimensional spaces has been extensively studied in \cite{CB74}. 
\begin{proposition}\label{prop:logconc}
A $p$-exponential measure $\pe$ is logarithmically-concave. That is, for any measurable sets $A, B\in\mathcal{B}(\R^\infty)$ and any $s\in[0,1]$ it holds \[\pe(sA+(1-s)B)\geq \pe(A)^s\cdot\pe(B)^{1-s}.\]
\end{proposition}

This is a straightforward result based on \cite{CB74}. A proof, done for a specific type of choice of $\gamma$ without loss of generality, can
be found in \cite[Lemma 3.4]{ABDH18}. Logarithmic concavity is a very strong property which for example implies unimodality, see \cite[Section 2]{ABDH18}. An immediate consequence is the following inequality called Anderson's inequality, implied by  \cite[Theorem 6.1]{CB74}, which holds since we consider centered measures. 

\begin{proposition}\label{prop:anderson}
Let $\pe$ be a $p$-exponential measure. For any closed, symmetric and convex set $A\subset \R^\infty$, we have \[\pe(A+x)\leq \pe(A), \forall x\in \R^\infty.\] 
\end{proposition}

Logarithmic-concavity also implies  the following zero-one law, see \cite[Theorem 4.1]{CB74}.

\begin{proposition}\label{prop:01}
Let $\pe$ be a $p$-exponential measure. Then for any linear subspace $V\subset \R^\infty$ we have that $\pe(V)=0$ or $1$.
\end{proposition}

\subsection{Absolute continuity}
We next consider the equivalence or singularity of a $p$-exponential measure to its translations.
\begin{definition}For a measure $\nu$ on a measurable space $(\X,\F)$, we define the space of admissible shifts $\sh=\sh(\nu)$ to be the subspace of all translations $h\in\X$ such that $\nu_h(\cdot):=\nu(\cdot-h)$ is equivalent to $\nu$ as measures. 
\end{definition}
The next proposition identifies the space of admissible shifts of the $p$-exponential measure and provides an expression for the Radon-Nikodym derivative of $\pe_h$ with respect to $\pe$, for $h\in\sh(\pe)$. It also shows that the two measures are singular for $h\notin\sh(\pe)$. 
\begin{proposition}
\label{p:shift}
Let $\pe$ be a $p$-exponential measure  and let $h\in\R^\infty$. Then $\pe_h$ and $\pe$ are either equivalent or singular. The space of admissible shifts of $\pe$ is
\[\sh=\sh(\pe)=\{h\in \R^\infty: \sum_{\ell=1}^\infty h_\ell^2\gamma_\ell^{-2}<\infty\}.\] In particular, it is a separable Hilbert space with norm 
\[\norm{h}_{\sh}=\Big(\sum_{\ell=1}^\infty h_\ell^2\gamma_\ell^{-2}\Big)^\frac12, \quad \forall h\in \sh.\]
Furthermore, for $h\in \sh(\pe)$, 
\[\frac{d\pe_h}{d\pe}(u)=\lim_{N\to\infty}\exp\left(\frac1p \sum_{\ell=1}^N\left(\left|\frac{u_\ell}{\gamma_\ell}\right|^p -\left|\frac{h_\ell-u_\ell}{\gamma_\ell}\right|^p\right) \right)\quad\quad in \;\;L^1(\R^\infty,\pe).\]
\end{proposition}
The last result is an immediate application of a more general result valid for scaled independent products of univariate distributions with finite Fisher information and everywhere positive density, see {Proposition \ref{p:shiftgen}} in the supplement below.

Even though the Radon-Nikodym derivative between a centered and a translated $p$-exponential measure involves weighted $\ell_p$-type terms, the space of admissible shifts, even for $p\neq 2$, is a weighted $\ell_2$ space, that is a Hilbert space. 
Furthermore, it is straightforward to check that $\pe(\sh(\pe))=0$. Indeed, for $u$ drawn from a $p$-exponential measure we have $\norm{u}_\sh^2=\sum_{\ell=1}^\infty \xi_\ell^2$ which is almost surely infinite by the law of large numbers. 
 
Motivated by the exponent of the Radon-Nikodym derivative above, we define the following subspace.

\begin{definition}\label{defn:om}
For a $p$-exponential measure $\pe$, we define the separable Banach space \[\om=\om(\pe)=\{h \in\R^\infty: \sum_{\ell=1}^\infty |h_\ell|^p\gamma_\ell^{-p}<\infty\},\]
with norm \[\norm{h}_{\om}=\Big(\sum_{\ell=1}^\infty |h_\ell|^p \gamma_\ell^{-p}\Big)^\frac1p, \quad \forall h\in\om.\]
\end{definition}
The space $\om(\pe)$ is a weighted $\ell_p$ space which, since $p\in[1,2]$ and $\gamma_\ell$ is a decaying sequence, is continuously embedded in $\sh(\pe)$. 
 Clearly we also have that $\pe(\om(\pe))=0$. 

When working in a separable Banach space $X$ possessing a Schauder basis, the subspaces $\sh(\pe)\subset \R^\infty$ and $\om(\pe)\subset \R^\infty$ are naturally identified with subspaces of $X$. If $X$ is a function-space, then $\om$ and $\sh$ correspond to subspaces of functions of higher regularity.  
In the Gaussian case $p=2$, we have that $\sh$ and $\om$  are identified with the RKHS \cite[Section I.6]{GV17}, but in general the two spaces differ and have different roles. 

For Gaussian measures, the RKHS is compactly embedded in any separable Banach space $X$ of full measure, \cite[Proposition 2.6.9]{GN16}. The next proposition generalizes this statement for  $p$-exponential measures. It follows from \cite[Theorem 5.1.6.]{Boga10}, which holds for general Radon measures on locally convex spaces. 

\begin{proposition}\label{prop:comp}
Let $\pe$ be a $p$-exponential measure on a separable Banach space $X$ with a Schauder basis. The space of admissible shifts $\sh(\pe)$ is compactly embedded into $X$. As a consequence, $\om(\pe)$ is also compactly embedded into $X$. 
\end{proposition}

\subsection{Support and concentration}\label{ssec:conc}

In this subsection, $(X,\norm{\cdot}_X)$ is a separable Banach space possessing a Schauder basis and $\pe$ is a $p$-exponential measure on $X$, $\pe(X)=1$, defined by randomizing the coefficients of random series expansions in the Schauder basis as explained in Subsection \ref{ssec:defn}. For a Gaussian measure on $X$, it is known that its topological support is the closure of the RKHS in $X$, \cite[Corollary 2.6.17]{GN16}. 
We next show an analogous result for $p$-exponential measures. Since $\gamma$ is a sequence of positive scalings, $p$-exponential measures are non-degenerate, that is their support is the whole space $X$.

\begin{proposition}\label{prop:supp2}
Let $\pe$ be a $p$-exponential measure on $X$. Then
\[X=supp(\pe)=\overline{\sh}^{\norm{\cdot}_X}=\overline{\om}^{\norm{\cdot}_X},\] \mods{where $\overline{\sh}^{\norm{\cdot}_X}, \overline{\om}^{\norm{\cdot}_X}$ denote the closures in the norm $\norm{\cdot}_X$ of the spaces $\sh, \om$, respectively.}
\end{proposition}

The role of the subspace $\om$ is revealed in the next two results, which study the probability of non-centered balls in $X$ relative to the probability of centered ones, under a $p$-exponential measure. Proposition \ref{prop:anderson}, showed that for fixed radius $\epsilon>0$, there is a loss of probability when shifting from centered  to non-centered balls. In the next proposition we prove a lower bound on the loss of probability when the shift is in the space $\om$. 

\begin{proposition}\label{prop:lbd}
Let $\pe$ be a $p$-exponential measure on  $X$. Then for $h\in \om$ and any $\epsilon>0$, we have 
\[\pe(\epsilon B_X+h)\geq e^{-\frac1p\norm{h}_{\om}^p}\pe(\epsilon B_X).\]
\end{proposition}

Our proof relies on certain properties of the function $|\cdot|^p, p\in[1,2]$ appearing in the exponent of the Radon-Nikodym derivative ${d\pe_h}/{d\pe}$ in Proposition \ref{p:shift}, namely its symmetry together with its convexity and the concavity of its derivative on the positive semi-axis. 

{For $p=2$ we recover the Gaussian result \cite[Lemma 5.2]{VZrkhs}. For $p\neq2$, the loss of probability is exponential in the $\om$-norm and not in the Hilbert space norm of the space of admissible shifts $\sh$. As we will see in the next section, this adds a degree of difficulty to the study of posterior contraction for $p$-exponential priors. Due to the form of the Radon-Nikodym derivative in Proposition \ref{p:shift}, the last result is not surprising. In particular, it is consistent with the form of the Onsager-Machlup functional, that is the functional giving the most probable paths,  for Besov-space measures with $p=1$ in \cite[Theorem 3.9]{ABDH18}.}

We next extend the last lower bound to centers that are not necessarily in $\om$, using approximation. We restrict to centers in the topological support of $\pe$, $X$, since otherwise a small enough ball around $w$ has zero probability.
As in the Gaussian case, see \eqref{eq:gaussian_conc} in Subsection \ref{ssec:gauss}, we define the concentration function of a $p$-exponential measure. 

\begin{definition}\label{defn:conc}
Let $w\in X$. 
 We define the concentration function of the $p$-exponential measure $\pe$ on $X$ to be
\begin{equation*}
\varphi_{w}(\epsilon)=\inf_{h\in \om:\norm{h-w}_X\leq\epsilon} \frac1p\norm{h}_{\om}^p-\log \pe(\epsilon B_X).
\end{equation*} 
\end{definition}

The first term relates to approximation of the center $w\in X=\overline{\om}^{\norm{\cdot}_X}=\overline{\sh}^{\norm{\cdot}_X}$ by elements of the space $\om$. Unlike the Gaussian case and consistently with Proposition \ref{prop:supp2}, for $p\neq2$ this approximation does not take place in a Hilbert space. 
For any $w\in X$, since the $\om$-norm is convex and $p$-exponential measures are logarithmically-concave and non-degenerate, the concentration function is a strictly decreasing and convex function on the positive semi-axis. This follows very similarly to the Gaussian case see \cite[Lemma 3]{IC08} or the more readily adaptable \cite[Lemma I.26]{GV17}. In particular, the concentration function is continuous and blows-up as $\epsilon\to0$. Depending on the position of  $w\in X$ relative to the space $\om$, the blow-up rate is determined by the first or second term. For example, if $w\in \om$ the first term remains bounded and only the second term blows-up.

The interpretation of the concentration function is similar to the Gaussian case. For $w=0$, the first term is zero and $\varphi_0(\epsilon)$ measures the probability with respect to $\pe$ of a centered ball of radius $\epsilon$ in $X$. For $w\in X\setminus\{0\}$, the next theorem shows that the concentration function gives a lower bound on the probability of a ball of radius $\epsilon$ in $X$ around $w$, with the first term measuring the loss of probability due to moving the ball away from the origin. 

\begin{theorem}\label{thm:cf}
For any $w\in X$ we have that 
\begin{equation*}
-\log\pe(w+\epsilon B_X)\leq \varphi_w(\epsilon/2), \;\forall \epsilon>0.
\end{equation*}
\end{theorem}
The proof of the last theorem is very similar to the first part of the proof of \cite[Lemma 5.3]{VZrkhs}. It follows from Proposition \ref{prop:lbd} using the triangle inequality and approximation of $w\in X$ in $\om$.
\begin{remark}\label{rem:conc}
In the Gaussian case, the concentration function yields both an upper and a lower bound on the probability of small balls around a $w\in X$ \cite[Lemma 5.3]{VZrkhs}. While
the last theorem achieves a lower bound, it would be interesting to also prove an upper bound in the $p$-exponential case. However, the lack of inner product structure in the Radon-Nikodym derivative between a centered $p$-exponential measure and its translation makes this task considerably harder.
\end{remark}

The following inequality generalizes Borell's inequality which studies the concentration of Gaussian measures, \cite[Theorem 3.1]{CB75}. It is based on a sharp two level concentration inequality due to Michel Talagrand \cite[Theorem 2.4]{TA94}.

\begin{proposition}\label{p:tal}
Let $\pe$ be a $p$-exponential measure in $X$. Recall that $B_\om$ and  $B_\sh$ denote the closed unit balls in the spaces $\om$ and $\sh$, respectively. Then there exists a constant $K>0$ depending only on $p$, such that for any set $A\in\mathcal{B}(X)$ and any $r>0$ it holds \begin{equation}\label{eq:tal}\pe(A+r^\frac{p}2B_\sh+rB_\om)\geq 1-\frac{1}{\pe(A)}\exp\left(-\frac{r^p}K\right).\end{equation}
\end{proposition}

  Letting $A=\epsilon B_X$ for a fixed small $\epsilon>0$, the last inequality implies that while both $\om, \sh$ are null sets of $\pe$, the bulk of the mass of $\pe$ is contained in a small $\epsilon$-cushion in $X$ around the sum of a ball of radius $r$ in $\om$ and a ball of radius $r^\frac{p}2$ in $\sh$, for $r$ large. 
 This interpretation is similar to the one for Borell's inequality presented in the discussion after \cite[Proposition 11.17]{GV17}, which is simpler since in the Gaussian case $\om=\sh$. 
     
\begin{remark}
Borell's inequality \cite[Theorem 3.1]{CB75} for Gaussian measures has the form of a stronger isoperimetric inequality, which in turn implies the concentration inequality \eqref{eq:tal} in the case $p=2$ and $\sh=\om$. Using results in isoperimetry for finite independent products of standard  univariate $p$-exponential distributions \cite{CR10}, together with the techniques in \cite{CB75} to pass from finite to infinite dimensions, one can show that there exists $K=K(p)>0$ such that for any $A\in\mathcal{B}(X)$ it holds \[\pe(A+rB_\sh)\ge F_p(F_p^{-1}(\pe(A))+Kr),\] where $F_p$ is the cumulative distribution function of the univariate standard $p$-exponential distribution. 
The concentration inequality implied by the above inequality has the form \[\pe(A+rB_Q)\geq 1-\frac1{\mu(A)}\exp\left(-\frac{r^p}K\right),\] and for $p\in[1,2)$ is strictly weaker than the one in Proposition \ref{p:tal}, since it involves balls of radius $r$ in the space $\sh$ which strictly contains $\om$.

\end{remark}

\section{General contraction theorem for $p$-exponential priors}\label{sec:main}

We next state our general contraction result for $p$-exponential priors in a separable Banach space $X$ possessing a Schauder basis, which generalizes the Gaussian contraction result \cite[Theorem 2.1]{VZ08}. It shows that for a $p$-exponential prior and a $w_0\in X$, if $\epsilon_n$ is such that the blow-up rate of the concentration function satisfies \begin{equation}\label{eq:blow}\varphi_{w_0}(\epsilon_n)\leq n\epsilon_n^2,\end{equation} then there exist sets $X_n\subset X$ of bounded complexity containing the bulk of the prior mass, and the prior puts sufficient mass around $w_0$. These assertions 
are in accordance with the requirements of results giving upper bounds on the contraction rate at $w_0$ for general priors, 
see the discussion in Subsection \ref{ssec:gauss} and the results in Section \ref{sec:mod} below. 

To prove our contraction result, we follow the techniques of the proof of the Gaussian  result \cite[Theorem 2.1]{VZ08}, which is based on Borell's inequality \cite[Theorem 3.1]{CB75} together with the concentration function and its relation to lower bounds on the probability of shifted small balls \cite[Lemma 5.3]{VZrkhs}. However, the situation for $p$-exponential priors is more complicated, due to the intricate form of the available concentration inequality in Proposition \ref{p:tal}. In particular, due to the fact that for $p\in[1,2)$,  the concentration inequality \eqref{eq:tal} involves both balls in $\sh$ and balls in $\om$, while the decentering result in Proposition \ref{prop:lbd} refers to elements in $\om$, in order to prove the complexity bound we need to approximate elements in $\sh$ by elements in $\om$. To this end we let $f,g: \R_{>0}\to\R_{>0}$ be two respectively non-decreasing and non-increasing functions, such that for $\epsilon,a>0$ and for any $h\in aB_\sh$ it holds
\begin{equation}\label{eq:inf}\inf_{x\in\om:\norm{x-h}_X\leq\epsilon}\norm{x}_\om^p\leq f(a) g(\epsilon)^{1-\frac{p}2}\end{equation}and  as $a\to\infty$, $f(a)$ grows at most polynomially to infinity.
For $p=2$, since $\sh=\om$, we can choose $f(a)=a^2$ while $g$ is redundant. For $p\in[1,2)$, since $\om\subsetneq \sh$, $g$ needs to satisfy $g(\epsilon)\to\infty$ as $\epsilon\to0$. For optimal results we need to choose $f$ and $g$ so that the bound \eqref{eq:inf} is as tight as possible. As a result of this extra approximation step, we get a more complicated form on the right hand side of the complexity bound, see \eqref{eq:adh22} below, compared to the Gaussian case \cite[Theorem 2.1]{VZ08}. \mods{Note that the factorization of the right hand side of \eqref{eq:inf} into the two functions $f$ and $g$, is not important for the theory, but arises naturally in practice, see Lemmas \ref{lem:l2dom} and \ref{lem:l8dom} below.}

\begin{theorem}\label{t:main}
Let $\pe$ be a $p$-exponential measure with scaling sequence $\gamma$ in a separable Banach space $X$ with Schauder basis, where $p\in[1,2]$.
Let $W\sim \pe$.  Fix $f,g : \R_{>0}\to \R_{>0}$, as in \eqref{eq:inf} above and let $w_0\in X$. 

 Assume $\epsilon_n>0$ such that $\varphi_{w_0}(\epsilon_n)\leq n\epsilon_n^2$, where $n\epsilon_n^2\gtrsim1$. Then for any $C>1$, there exists a measurable set $X_n\subset X$ and a constant $R>0$ depending on $C, p$ and $f$, such that 
\begin{align}\label{eq:adh22}
\log N(4{\epsilon}_n, X_n,\norm{\cdot}_X)&\leq R\big(n\epsilon_n^2\vee f(n^\frac12\epsilon_n)g(\epsilon_n)^{1-\frac{p}2}\big),\\[10pt]
\label{eq:adh23}
\rp(W\notin X_n)&\leq \exp(-Cn\epsilon_n^2),\\[10pt]
\label{eq:adh24}
\rp(\norm{W-w_0}_X<2\epsilon_n)&\geq \exp(-n\epsilon_n^2).
\end{align} 
\end{theorem}

The difference between the assertions of the above theorem compared to the Gaussian result \cite[Theorem 2.1]{VZ08}, is  the right hand side in the complexity bound \eqref{eq:adh22}, which is potentially larger than $n\epsilon_n^2$, depending on which of the two terms dominates in the maximum asymptotically as $\epsilon_n\to0$. 
In the Gaussian case $p=2$, the right hand side in  \eqref{eq:adh22} becomes $n\epsilon_n^2$ and we recover \cite[Theorem 2.1]{VZ08}. For $p\in[1,2)$, depending on the norm in the parameter space $X$, we have a different form of the tightest functions $f$ and $g$ that we can verify to satisfy \eqref{eq:inf}. If the quality of approximation in $\norm{\cdot}_X$ of elements in $\sh$ by elements in $\om$ is not sufficiently good, the right hand side in \eqref{eq:adh22} can be dominated by the second term and in this case the complexity bound is not in accordance with the corresponding complexity bound in general prior contraction results like \cite[Theorem 8.9 and 8.19]{GV17}. The next corollary handles such situations.

\begin{corollary}\label{c:main}
Under the assumptions of Theorem \ref{t:main}, let $\tilde\epsilon_n>0$ be such that \begin{equation}\label{eq:sub}f(n^\frac12{\epsilon}_n)g({\epsilon}_n)^{1-\frac{p}2}\lesssim n\tilde{\epsilon}_n^2.\end{equation}
Then for any $C>1$, there exists a measurable set $X_n\subset X$ and a constant $R>0$, such that 
\begin{align}\label{eq:adh22'}
\log N(4({\epsilon}_n\vee\tilde\epsilon_n), X_n,\norm{\cdot}_X)&\leq Rn(\epsilon_n\vee\tilde\epsilon_n)^2,\\[10pt]
\label{eq:adh23'}
\rp(W\notin X_n)&\leq \exp(-Cn\epsilon_n^2),\\[10pt]
\label{eq:adh24'}
\rp(\norm{W-w_0}_X<2\epsilon_n)&\geq \exp(-n\epsilon_n^2).
\end{align}

\end{corollary}

The corollary follows immediately from Theorem \ref{t:main}, since taking a larger $\epsilon_n$ makes the left hand side of the complexity bound \eqref{eq:adh22} smaller and the right hand side larger. 

In settings for which \eqref{eq:sub} is satisfied with $\tilde{\epsilon}_n=\epsilon_n$, for $\epsilon_n$ the fastest rate solving \eqref{eq:blow}, we can apply the corollary and the resulting three assertions are in accordance with the general contraction results which show that this $\epsilon_n$ is an upper bound on the contraction rate. We will see in Section \ref{sec:l2}, that this is the case in separable Hilbert space settings for $\alpha$-regular $p$-exponential priors and under Besov-type regularity of $w_0$. In this situation the intuition about the contraction rate is similar to the Gaussian case, the only difference being that the RKHS is replaced by the Banach space $\om$. We refer to the discussion in  \cite[Section 11.3]{GV17} which we adapt here to $p$-exponential priors:  the rate of contraction is up to constants the maximum of the minimal solution to the small ball inequality 
\[-\log\pe(\epsilon_nB_X)\leq n \epsilon_n^2\] and the minimal solution to the approximation inequality \[\inf_{h\in \om:\norm{h-w}_X\leq\epsilon_n} \norm{h}_{\om}^p \leq n\epsilon_n^2.\] The first inequality depends only on the prior, showing that priors that put little mass around the origin give slow rates independently of $w_0$. The second inequality depends on both the prior and the true $w_0$  and relates to the loss of probability mass in small balls around $w_0$ compared to centered small balls. It shows that even if the prior puts a lot of mass around the origin, it is still possible to give a slow rate at a $w_0$, depending on the positioning of $w_0$ relative to the Banach space $\om$. 

On the other hand in settings for which \eqref{eq:sub} is only satisfied for $\tilde{\epsilon}_n$ a sequence decaying more slowly than the fastest rate $\epsilon_n$ solving \eqref{eq:blow}, the resulting three assertions of the corollary are only in accordance with the general contraction result in the independent and identically distributed data case \cite[Theorem 8.9]{GV17}, which shows that the slower rate  $\tilde{\epsilon}_n$ is an upper bound on the contraction rate. We will see in Section \ref{sec:l8} that such issues arise for $\alpha$-regular $p$-exponential priors in $C[0,1]$, defined via wavelet bases. {In this case the intuition regarding the rates of contraction is obfuscated.}

\begin{remark} For Gaussian priors, the availability of an upper bound on the probability of small balls around an element $w\in X$ in terms of the concentration function, enabled the study of lower bounds on posterior contraction rates in \cite{IC08}. Such an upper bound remains open for $p$-exponential priors with $p\neq2$, see remark \ref{rem:conc}, hence the use of the techniques of \cite{IC08} to similarly obtain lower bounds on posterior contraction rates in this case is precluded.
\end{remark}


\section{Posterior contraction for specific models}\label{sec:mod}
We next use the results of the preceding section to study posterior contraction for general $p$-exponential priors in specific nonparametric statistical settings. 
Indeed, 
 the assertions of Theorem \ref{t:main} and Corollary \ref{c:main}, point to the assumptions of the well known general model and general prior posterior contraction rate results \cite[Theorem 8.9 and 8.19]{GV17}. However, the former results are expressed purely in terms of the Banach space norm of the parameter space, while the latter have conditions relating to statistically relevant norms and discrepancies. As discussed in Subsection \ref{ssec:gauss}, the necessary reconciliatory work has already been carried out in various standard statistical settings and can be readily used for $p$-exponential priors in the same way that it was used for Gaussian priors in \cite[Section 3]{VZ08}.

Note, that compared to the Gaussian contraction results found in \cite[Section 3]{VZ08} or \cite[Section 11.3]{GV17}, in the formulation of our results we need to take into account the more complicated complexity bound in \eqref{eq:adh22}. For reasons of brevity, we only present here contraction  results for density estimation and for the white noise model. Results in other models such as binary classification and nonparametric regression follow similarly.

\subsection{Density estimation}\label{ssec:de}
We consider the estimation of a probability density $\dens$ relative to a finite measure $\nu$ on a measurable space $(\ms,\mss)$, based on a sample of observations $X_1,\dots,X_n|\dens\stackrel{iid}{\sim}\dens$. Following  \cite[Section 11.3.1]{GV17}, we construct a prior $\Pi$ on $\dens$ by  letting \[\dens(x)=\frac{e^{W(x)}}{\int_{\ms}e^{W(y)}d\nu(y)}, \;x\in\ms,\] where $W$ is a draw from a $p$-exponential measure $\mu$ on $L_\infty(\ms)\cap C(\ms)$.  We require that $W$ is almost surely continuous so that it can be evaluated at $x\in\ms$ and $\dens(x)$ is well defined. We can define $p$-exponential priors with continuous and bounded paths, see Section \ref{sec:l8} below.

Let $\Pi_n(\cdot|X_1,\dots,X_n)$ be the posterior distributions after observing $X_1,\dots, X_n$. The following contraction result is a generalization of the Gaussian result \cite[Theorem 3.1]{VZ08}. It gives contraction rates in the Hellinger distance $d_H(\cdot,\cdot)$ between two probability densities. The proof is identical to the Gaussian case, once we take into account Corollary \ref{c:main} (see also \cite[Theorem 11.21]{GV17}). 

\begin{theorem}\label{thm:de}
Let $W$ be a $p$-exponential random element in a separable Banach subspace of $L_\infty(\ms)$ possessing a Schauder basis,
which is almost surely continuous. Assume $w_0=\log \dens_0$ belongs to the support of $W$ and denote by $P^n_0$ the corresponding distribution of the vector $(X_1,\dots,X_n)$. Let $\epsilon_n$ satisfying \eqref{eq:blow} with respect to $\norm{\cdot}_{L_\infty}$ and $\tilde\epsilon_n$ satisfying \eqref{eq:sub} where the functions $f, g$  are defined in \eqref{eq:inf}. Then $\Pi_n(\dens:d_H(\dens,\dens_0)>M(\epsilon_n\vee\tilde\epsilon_n)|X_1,\dots,X_n)\to0,$ in $P^n_{0}$-probability, for some sufficiently large constant $M$. 
\end{theorem}

\subsection{White noise model}\label{ssec:wn}
We study the estimation of a signal $w\in\Theta\subset L_2[0,1]$, from the observation of a sample path of the stochastic process \[X_t^{(n)}=\int_0^t w(s)ds+\frac1{\sqrt{n}}B_t, \;t\in[0,1]\] where $B$ is standard Brownian motion. Let $P_w^{(n)}$ be the distribution of the sample path $X^{(n)}$ in $C[0,1]$. As a prior on $w$ we take a  $p$-exponential random element $W$ in $L_2[0,1]$. We can define such priors using random series expansions for example in an orthonormal basis of $L_2[0,1]$, see Section \ref{sec:l2} below. Observe that by Proposition \ref{prop:supp2}, the topological support of such a prior is $L_2[0,1]$.

We denote by $\Pi_n(\cdot|X^{(n)})$ the posterior on $w$ after observing the sample path $X^{(n)}$.
The following posterior contraction result is a generalization of \cite[Theorem 3.4]{VZ08}. The proof is identical to the Gaussian case, and follows immediately by combining \cite[Theorem 8.31]{GV17} and Theorem \ref{t:main}. 

\begin{theorem}\label{thm:wn}
Let $W$ be a $p$-exponential random element in $L_2[0,1]$. Assume that the true value of $w$ is contained in the support of $W$, $w_0\in L_2[0,1]$. Furthermore, assume that $\epsilon_n$ satisfies the rate equation $\varphi_{w_0}(\epsilon_n)\leq n\epsilon_n^2$ with respect to the $L_2[0,1]$-norm  and is such that  \eqref{eq:adh22} holds with $n\epsilon_n^2$ on the right hand side. Then $\Pi_n(w:\norm{w-w_0}_{L_2}>M\epsilon_n|X^{(n)})\to0$ in $P^{(n)}_{w_0}$-probability, for some $M>0$.
\end{theorem}

\section{The separable Hilbert space setting}\label{sec:l2}
In this section we consider $p$-exponential measures in a separable Hilbert space $X$. Since any separable Hilbert space is isometrically isomorphic to the space of square summable sequences $\ell_2$, we can equivalently, as far as concentration is concerned, work in $\ell_2$. This equivalence holds, provided the $p$-exponential measure in $X$ is defined using expansions in an orthonormal basis, see Subsection \ref{ssec:defn}. In particular, we consider $\alpha$-regular $p$-exponential measures in sequence space and study their concentration at centers of varying Besov-type regularity. Note that these measures are merely a different parametrization of Besov-space priors used in applied inverse problems literature, \cite{LSS09}. We combine our findings with Theorem \ref{thm:wn}, to obtain posterior contraction rates in the white noise model under Besov-type regularity of the truth.

We first define the following Besov-type sequence spaces.

\begin{definition}\label{defn:besov}
For any $s>0$, $q\geq1$ and $d\in \N$, we denote by $B^s_q$ the weighted $\ell_q$ spaces \[B^s_q={B^s_q(d)}=\{u\in\R^\infty: \sum_{\ell=1}^\infty \ell^{q(\frac{s}d+\frac12)-1}|u_\ell|^q<\infty\},
\quad \norm{u}_{B^s_q}=\left( \sum_{\ell=1}^\infty \ell^{q(\frac{s}d+\frac12)-1}|u_\ell|^q\right)^\frac1q.\]
\end{definition}
The case $q=2$ corresponds to Sobolev-type spaces. These spaces can be identified for example to Sobolev spaces $H^s$ of periodic functions on $\mathbb{T}^d=(0,1]^d$  with $s$ square integrable derivatives, using expansions in the Fourier basis. Similarly, for $q\neq2$ $B^s_q$ can be identified with the Besov space $B^s_{q_1q_2}$ of periodic functions on $\mathbb{T}^d$, with integrability parameters $q_1=q_2=q$ and smoothness parameter $s$, using expansions in certain sufficiently regular orthonormal wavelet bases \cite{schmeisser1987topics}. {Of particular interest is the case $q<2$, which includes classes of non-smooth and spatially inhomogeneous functions, see for example \cite{DJ98}.} {Such functions are useful in many scientific disciplines, for example in geophysics and medical imaging, as they can be smooth in one area and rough in another one. The rates we obtain below suggest that, when interested in reconstructing an unknown function of this type, it is beneficial to use a non-Gaussian $p$-exponential prior.} 

\subsection{$\alpha$-regular $p$-exponential priors in $\ell_2$}
Consider $\pe$ a $p$-exponential measure in sequence space with $\gamma_\ell= \ell^{-\frac12-\frac{\alpha}d}, \alpha>0, p\in[1,2]$, $d\in\N$. {The parameter $d$ expresses the dimension inherent in the Hilbert space $X$ and is fixed for a given model, for example $X=L_2(\mathbb{T}^d)$.} As discussed in Section \ref{ssec:defn}, since $\gamma\in \ell_2$ it holds $\pe(\ell_2)=1$. 
Furthermore, by Proposition \ref{prop:supp2} the support of $\pe$ is $\ell_2$. We call such a measure a \emph{$d$-dimensional $\alpha$-regular $p$-exponential measure in $\ell_2$.}

The next result studies the Besov-type regularity of draws from $\pe$ and  justifies the name \emph{$\alpha$-regular}.

 \begin{lemma}\label{lem:bessup}
Assume $\pe$ is a $d$-dimensional $\alpha$-regular $p$-exponential measure in $\ell_2$. Then for any $q\geq 1$, we have that $\pe(B^s_q)=1$ for all $s<\alpha,$ and $\pe(B^s_q)=0$ for all $s\geq \alpha$. 
\end{lemma}

We next study the concentration function $\varphi_w(\cdot)$ of $\pe$, defined for centers $w\in \ell_2$; see Definition \ref{defn:conc} where $X=\ell_2$.
The following lemma identifies the space $\om$ in which we approximate the center $w\in\ell_2$ in the first term of $\varphi_w$, as well as the shift space $\sh$. It follows immediately from Proposition \ref{p:shift} and Definitions \ref{defn:om} and \ref{defn:besov}. 

\begin{lemma}\label{lem:oml2}
Assume $\pe$ is a $d$-dimensional $\alpha$-regular $p$-exponential measure in $\ell_2$.  
Then $\om=\oma:=B^{\alpha+\frac{d}p}_p$ and $\sh=\sha:=B^{\alpha+\frac{d}2}_2$. 
\end{lemma}

We next study the blow-up rate of the concentration function $\varphi_w(\epsilon)$ as $\epsilon\to0$ in the present setting. In particular, we find upper bounds on the minimal solution $\epsilon_n$ of the inequality  $\varphi_{w_0}(\epsilon_n)\leq n\epsilon_n^2$ depending on the Besov-type regularity of $w_0$. 

\begin{proposition}\label{prop:ubl2}
Assume that $\pe$ is a $d$-dimensional $\alpha$-regular $p$-exponential measure in $\ell_2$  and that $w_0\in B^\beta_q$ for $\beta>0\vee (\frac{d}q-\frac{d}2)$, $q\geq1$, $p\in[1,2]$. Then as $n\to\infty$
the rate $\epsilon_n\asymp r_n^{\alpha,\beta,p, q}$ satisfies the inequality $\varphi_{w_0}(\epsilon_n)\leq n\epsilon_n^2$, for constants which depend on $w_0$ only through its $B^\beta_q$ norm and
for $ r_n^{\alpha,\beta,p, q}$ given below:

\begin{enumerate}
\item[i)]For $q\geq 2$ 
\[r_n^{\alpha,\beta,p,q}:=\left\{\begin{array}{ll}n^{-\frac{\beta}{d+2\beta+p(\alpha-\beta)}}, &  
\;\mbox{if $\alpha\geq \beta$,} 
                                     \\ n^{-\frac{\alpha}{d+2\alpha}}, &  \;\mbox{if $\alpha<\beta$}.
                                    \end{array}\right.\]
\item[ii)]For $q<2$ and $p\leq q$, letting $a=\sqrt{2\beta d p+\beta^2p^2+d^2(1+2p-4p/q)}$
\[r_n^{\alpha,\beta,p,q}:=
	\left\{\begin{array}{ll}n^{-\frac{2\beta q+d(q-2)}{4d(q-1)+4\beta q+2pq(\alpha-\beta)}}, &  
\;\mbox{if $\alpha\geq \frac{\beta p -d+a}{2p}$,} 
                                     \\ n^{-\frac{\alpha}{d+2\alpha}}, &  \;\mbox{if $\alpha<\frac{\beta p -d+a}{2p}$}.
                                    \end{array}\right.\]
\item[iii)]For $q< 2$ and $p>q$, letting $a=\sqrt{\frac{2\beta dq(2q-p)+\beta^2pq^2+d^2(p+2q^2-4q)}{p}}$
\[r_n^{\alpha,\beta,p,q}:=
	\left\{\begin{array}{ll}n^{-\frac{2\beta q+d(q-2)}{2d(p+q-2)+4\beta q+2pq(\alpha-\beta)}}, &  
\;\mbox{if $\alpha\geq \frac{\beta q-d+a}{2q}$,} 
                                     \\ n^{-\frac{\alpha}{d+2\alpha}}, &  \;\mbox{if $\alpha<\frac{\beta q-d+a}{2q}$}.
                                    \end{array}\right.\]
\end{enumerate}
\end{proposition}

The proof of the last proposition, builds on Lemmas \ref{lem:l2sb} and \ref{lem:infl2} in Subsection \ref{sec:conc} below, in which we estimate the two terms of the concentration function. Notice that the assumption $\beta>0\vee(\frac{d}q-\frac{d}2)$ secures that $w_0\in\ell_2$, see {Lemma G.1 in the supplement below.}

Consider a nonparametric inference problem in a separable Hilbert parameter space $X$, where the $X$-norm relates suitably to the statistically relevant norms for the model and there exist exponentially powerful tests for separating the truth from balls in $X$ at a certain distance from it. In the assumed separable Hilbert setting, we can verify that for $\epsilon_n$ the rate in Proposition \ref{prop:ubl2},  the maximum appearing in the right hand side of the complexity bound \eqref{eq:adh22} in our general contraction Theorem \ref{t:main}, is dominated by $n\epsilon_n^2$; see Lemma \ref{lem:l2dom} in Subsection \ref{sec:conc}. Together with Proposition \ref{prop:ubl2} and Theorem \ref{t:main}, this suggests that if we use as prior a $d$-dimensional $\alpha$-regular $p$-exponential measure in $\ell_2$ identified with a measure on $X$ via a series expansion in an orthonormal basis of $X$, then $r_n^{\alpha,\beta,p, q}$ is an upper bound on the posterior contraction rate when the truth belongs to  $B^\beta_q$. Here we identify Besov regularity in $X$ with Besov regularity of the sequence of coefficients in the orthonormal basis. 
For example, in the white noise model combining the last two results with Theorem \ref{thm:wn}, we get immediately the next result for $\alpha$-regular $p$-exponential priors in $L_2[0,1]$ ({here $d=1$}).

\begin{theorem}\label{thm:wn2}
Consider the white noise model of Subsection \ref{ssec:wn}, and let $\Pi=\pe$ be an $\alpha$-regular $p$-exponential prior in $L_2[0,1]$, $\alpha>0, p\in[1,2]$. Assume $w_0\in B^\beta_q$, ${\beta>0\vee (\frac1q-\frac12)}, q\geq1$ and let $r_n^{\alpha,\beta,p,q}$ be defined as in Proposition \ref{prop:ubl2}. Then  for $M$ large enough, 
 as $n\to\infty$
\[\Pi_n( w\in L_2[0,1]: \norm{w-w_0}_{L_2}\leq Mr_n^{\alpha,\beta,p, q}|X^{(n)})\to1,\] in $P^{(n)}_{w_0}$-probability. 
\end{theorem}

Notice that since in Proposition \ref{prop:ubl2} all the constants depend on $w_0$ only through its $B^\beta_q$-norm, the rates of contraction in the above theorem hold uniformly for $w_0$ in $B^\beta_q$-balls.

The last result generalizes existing contraction results in the conjugate setting of the white noise model with Gaussian priors and under Sobolev-type regularity of the truth, $p=q=2$; see \cite[Theorem 5.1]{LZ00} and \cite[Theorem 2.1]{BG03}, as well as \cite[Theorem 2]{IC08} which discusses the sharpness of the Gaussian contraction rates. Note that in our setting, unless $p=2$, the $p$-exponential prior is non-conjugate to the Gaussian likelihood of the white noise model. However, explicit calculations are possible to get upper bounds on the rate of posterior contraction, see \cite[Corollary 3]{CN13}, for Sobolev-type regularity of the truth $q=2$, when $\alpha\leq\beta$. Our result agrees with the existing rates in both aforementioned special cases, but goes further and in particular studies rates of contraction under Besov-type smoothness, that is the more intricate case $q\neq2$. 

Minimax rates in $L_2$-loss, for function estimation in the white noise model under Besov-regularity, have been studied in \cite{DJ98}; see also \cite[Chapter 10]{HKPT98} for density estimation again in $L_2$-loss. The results there, show that for all {$q\geq1$ and for $\beta>1/q$ or $\beta\geq1$ when $q=1$, the minimax rate is 
\mods{\begin{equation}\label{eq:nlminimax}
m_n:=n^{-\frac{\beta}{1+2\beta}}. 
\end{equation}}
An interesting feature, is that for $q<2$ linear estimators do not achieve the minimax rate, and instead only achieve the rate
\mods{\begin{equation}\label{eq:lminimax}
l_n:=n^{-\frac{\beta-\gamma/2}{1+2\beta-\gamma}},
\end{equation} }where $\gamma:=\frac{2-q}q>0$. For $q\geq2$, linear estimators do achieve the minimax rate $m_n$. This change of behaviour is attributed to the fact that, for $q<2$, functions in $B^\beta_q$ are not in general spatially homogeneous, but instead can be irregular in some parts and smooth in other parts. As explained in \cite{DJ98}, linear estimators cannot cope well with this inhomogeneity and either oversmooth the irregular part, or undersmooth the smooth part, or both. 

\begin{remark}[Results for $q\geq2$]\label{rem:l2rates}
An inspection of the bounds in Theorem \ref{thm:wn2}, reveals that for $q\geq2$, the particular value of $q$ does not influence the contraction rate. When $\beta=\alpha$, we get the minimax rate $m_n$, \mods{see \eqref{eq:nlminimax}}, independently of $p\in[1,2]$. In the case of an undersmoothing prior, $\beta>\alpha$, the rates for  all $p\in[1,2]$ coincide and are slower than the minimax rate. Finally, for an oversmoothing prior, $\beta<\alpha$, the rate is faster the smaller $p$ is. This is reasonable, since for smaller $p$ there is a higher probability of $\xi_\ell$ in the definition of the $p$-exponential prior having large values, which counteracts the oversmoothing effect of the prior-scaling sequence $\gamma_\ell=\ell^{-\frac12-\alpha}$.
\end{remark}

\begin{remark}[Results for $q<2$]\label{rem:l2ratesqle2}
An inspection of the bounds in Theorem \ref{thm:wn2}, reveals that for $q<2$, the particular value of $q$ does influence the contraction rate and we do not get the minimax rate $m_n$ \mods{defined in \eqref{eq:nlminimax}}, for any admissible combination of $\alpha, \beta, p , q$. The best rates are achieved for $p=q$ and are better than the linear minimax $l_n$, \mods{see \eqref{eq:lminimax}}, while Gaussian priors appear to have the worst performance and to be limited by the linear minimax rate:
\begin{enumerate}
\item[i)] Let $q<2$, $p\leq q$ and fix $\beta>\frac1q$ or $\beta\geq1$ if $q=1$ {(so that the minimax and linear minimax rates in \cite{DJ98} hold).} For small $\alpha>0$, the rate is $n^{-\frac{\alpha}{1+2\alpha}}$ which is suboptimal but improves as $\alpha$ increases. Since the other leg of the bound deteriorates as $\alpha$ increases, to achieve the minimax rate $m_n$, the bound $n^{-\frac{\alpha}{1+2\alpha}}$ needs to hold for $\alpha$ all the way up to $\alpha=\beta$. An easy calculation shows that the switching point between the two legs, $\alpha=\frac{\beta p -1+a}{2p}$, is smaller than $\beta$ but larger than $\beta-\frac{\gamma}2$. This implies that we cannot achieve the minimax rate, but there are values of $\alpha$ for which we achieve rates faster than the linear minimax rate $l_n$. In fact, one can check that for larger $p$ (that is $p\leq q$ closer to $q<2$), the switching point gets closer to $\beta$, hence we can get closer to the minimax rate for suitable values of $\alpha$. However, even for $p=q$ we cannot reach the minimax rate.
\item[ii)] Let $q<2$, $p> q$ and fix $\beta>\frac1q$ or $\beta\geq1$ if $q=1$. The reasoning is the same as in item (i): the switching point in the rate, $\frac{\beta q -1+a}{2q}$, is for all $p\in(q,2]$ smaller than $\beta$, for $p\in(q,2)$ larger than $\beta-\frac{\gamma}2$ and for $p=2$ equal to $\beta-\frac{\gamma}2$. Hence we cannot achieve the minimax rate $m_n$ for any $p\in(q,2]$, for $p\in(q,2)$ there are values of $\alpha$ achieving rates faster than the linear minimax rate $l_n$ and for the Gaussian case, $p=2$, the best we can achieve is the linear minimax rate $l_n$. One can check that for $p$ smaller (that is $p$ closer to $q$), the switching point gets closer to $\beta$, hence we can get closer to the minimax rate for suitable values of $\alpha$. {Note, that at this stage it is not clear whether Gaussian priors are fundamentally limited by the linear minimax rate,} {see the discussion in Section H of the supplement below.}
\end{enumerate}
\end{remark}

{As detailed in the last remark, for $w_0\in B^\beta_q$ with $q<2$, for all $p\in[1,2]$, the best rate is achieved by an undersmoothing prior and this rate is not minimax. This is due to the infimum term in the concentration function dominating the centered small ball probability term, already for large enough $\alpha<\beta$. This motivates using \emph{rescaled} $p$-exponential priors in the next subsection, with a vanishing scaling as $n\to\infty$: the idea is to use an undersmoothing prior such that without rescaling the centered small ball probability term dominates in the concentration function, and choose the scaling to - in some sense - infuse additional regularity to the prior, in particular in order to balance the two terms of the concentration function. We will see that this can lead to minimax rates.}

\subsection{Rescaled $\alpha$-regular $p$-exponential priors in $\ell_2$}\label{ssec:resc}

Consider $\pel$ to be a rescaling of a $d$-dimensional $\alpha$-regular $p$-exponential prior $\pe$ in $\ell_2$, that is we take $\gammal_\ell=\lambda\gamma_\ell$ where $\gamma_\ell=\ell^{-\frac12-\frac{\alpha}d},$ for some $\lambda>0$. Then $\om=\oml, \sh=\shl$, where $\oml, \shl$ coincide with the spaces $\oma, \sha$ in Lemma \ref{lem:oml2}, but with norms $\norm{\cdot}_{\oml}=\lambda^{-1}\norm{\cdot}_{\oma}$ and $\norm{\cdot}_{\shl}=\lambda^{-1}\norm{\cdot}_{\sha}$, respectively. For any $w\in X$, the concentration function of $\pel$ is then given by
\begin{equation}\label{eq:rescf}
\varphil_{w}(\epsilon)=\lambda^{-p}\inf_{h\in \oma: \norm{h-w}_{\ell_2}\leq\epsilon}\norm{h}^p_{\oma}-\log\pe\left(\frac{\epsilon}{\lambda}B_{\ell_2}\right).
\end{equation}

{In the formulation of our results below, we let as in the discussion in the previous subsection
\begin{equation}\label{eq:minimax}m_n:=n^{-\frac{\beta}{d+2\beta}},\end{equation} 
which is the minimax rate of estimation of a $B^\beta_{q}(d)$ sequence under Gaussian white noise in $\ell_2$-loss, for all $q\geq1$ and for $\beta>\frac{d}q$ or $\beta\geq d$ when $q=1$, see \cite{DJ98}.}

We focus on the case $q<2$, since for $q\geq2$ we can achieve the minimax rate with $\alpha$-regular $p$-exponential priors, for $\alpha=\beta$ without rescaling, see Remark \ref{rem:l2rates}. For $q\geq2$, rescaling can still be beneficial, in particular in order to achieve (some degree of) adaptation to unknown smoothness $\beta$; for results for Gaussian priors see \cite{KSVZ12, SVZ13, RS17}. We will investigate adaptation with $p$-exponential priors in a separate study.

\begin{proposition}\label{prop:rescaledrate}
Assume that $\pel=\pel_n$ is a rescaled $d$-dimensional $\alpha$-regular $p$-exponential measure in $\ell_2$, $p\in[1,2]$, corresponding to $\lambda=\lambda_n$.  Let $w_0\in B^\beta_q$ for { $\beta>\frac{d}p\vee\frac{d}q$, $q\in [1,2)$}, \mods{and consider the corresponding minimax rate of estimation in Gaussian white noise in $\ell_2$-loss, $m_n$, as defined in \eqref{eq:minimax}}. Then as $n\to\infty$,
the rate $\epsilon_n\asymp \bar{r}_n$ satisfies the inequality $\bar\varphi_{w_0}(\epsilon_n)\leq n\epsilon_n^2$, for constants which depend on $w_0$ only through its $B^\beta_q$ norm, where $\bar{r}_n=\bar{r}_n^{\alpha,\beta,p, q}$ and $\lambda_n=\lambda_n^{\alpha, \beta, p ,q}$ are given below:
\begin{enumerate}
\item[i)] If $q=p$, $\alpha=\beta-\frac{d}p$, then $\bar{r}_n=m_n$ for $\lambda_n=n^{-\frac{d}{p(d+2\beta)}}$.
\item[ii)] If $q>p$, $\alpha=\beta-\frac{d}p$, then $\bar{r}_n=m_n\log^\frac{d(q-p)}{pq(d+2\beta)}n$ for $\lambda_n=n^{-\frac{d}{p(d+2\beta)}}\log^{\omega}n$, where $\omega=(p-\frac{2d}{d+2\beta})\frac{q-p}{p^2q}>0$.
\end{enumerate}
In all other combinations of $\alpha, \beta, p, q$, for  any choice of $\lambda_n$, $\bar{r}_n$ is polynomially slower than $m_n$. In particular, if $q<p$, then the best achievable rate is $\bar{r}_n=n^{\frac{d(p-q)-\beta p q}{2d(q-p)+2\beta p q+pqd}},$ for $\alpha=\beta-\frac{d}{q}$ and $\lambda_n=n^{-\frac{qd}{2qd+2\beta pq-2pd+pqd}}$.
\end{proposition}

{Notice, that the assumption $\beta>\frac{d}p\vee\frac{d}q$, is in place in order to secure that, for a fixed $\beta$, the values of $\alpha$ corresponding to the best achievable rates, are positive hence admissible, simultaneously for all combinations of $p, q$. This facilitates the comparison between the different choices of $p, q$. However, it is possible to compute (slower than minimax) rates with rescaled $p$-exponential priors for smaller values of $\beta>0\vee(\frac{d}q-\frac{d}2)$, for $\alpha>0$, with the same techniques used in the proof of Proposition \ref{prop:rescaledrate}.}

As in the previous subsection, we can verify that for $\epsilon_n$ the rate in Proposition \ref{prop:rescaledrate}, the maximum appearing in the right hand side of the complexity bound \eqref{eq:adh22} in our general contraction Theorem \ref{t:main}, is dominated by $n\epsilon_n^2$; see Lemma \ref{lem:l2domresc} in Subsection \ref{sec:conc}. This can be combined  with Proposition \ref{prop:rescaledrate} and Theorem \ref{t:main}, in order to get contraction rates in suitable nonparametric problems in separable Hilbert spaces, under rescaled $d$-dimensional $p$-exponential priors. In particular, in the white noise model, combining the last two results with Theorem \ref{thm:wn}, we get immediately  the next result for rescaled $\alpha$-regular $p$-exponential priors in $L_2[0,1]$ (here $d=1$).

\begin{theorem}\label{thm:wn3}
Consider the white noise model of Subsection \ref{ssec:wn}, and let $\Pi_n=\pel_n$ be a rescaled $\alpha$-regular $p$-exponential prior in $L_2[0,1]$, $\alpha>0, p\in[1,2]$. Assume $w_0\in B^\beta_q$, ${\beta>\frac{d}p\vee \frac{d}q}, q\in[1,2)$ and let $\bar{r}_n$ and $\lambda_n$ be defined as in Proposition \ref{prop:rescaledrate}. Then  for $M$ large enough, as $n\to\infty$
\[\Pi_n( w\in L_2[0,1]: \norm{w-w_0}_{L_2}\leq M\bar{r}_n|X^{(n)})\to1,\] in $P^{(n)}_{w_0}$-probability. 
\end{theorem}

\begin{remark}\label{rem:rem3}
The results of Theorem \ref{thm:wn3} show again that for $w_0\in B^\beta_q$ with $q\in[1,2)$, $\beta>\frac1p\vee\frac1q$, the best rates are achieved for $p=q$, while Gaussian priors appear to have the worst performance and to be limited by the linear minimax rate:
\begin{enumerate}
\item[i)] We can achieve the minimax rate $m_n$, \mods{see \eqref{eq:nlminimax}}, only with a $p$-exponential prior for $p=q$, for regularity $\alpha=\beta-\frac{1}p$ and for appropriate vanishing rescaling $\lambda_n$. For $p<q$ we can achieve the minimax rate $m_n$ up to logarithmic factors, again for regularity $\alpha=\beta-\frac{1}p$ and an appropriate vanishing rescaling. 
\item[ii)]For $q<p<2$, the best rate with rescaling is obtained for $\alpha=\beta-\frac{1}q$ and is improved compared to the best obtained rate without rescaling, but it remains polynomially slower than the minimax rate $m_n$, \mods{defined in \eqref{eq:nlminimax}}. The fact that the rate improves, is implicit in the proof of Proposition \ref{prop:rescaledrate}: it is shown that for any $\alpha\neq\beta-\frac{1}q$ the obtained rate with optimized rescaling is strictly slower than for $\alpha=\beta-\frac{1}q,$ while as discussed in Remark \ref{rem:l2ratesqle2}(ii), without rescaling the best rate is achieved for $\alpha>\beta-\frac{\gamma}2=\beta-\frac1q+\frac12$, in particular for $\alpha\neq \beta-\frac1q$.
\item[iii)] For the Gaussian case, $p=2$, rescaling does not improve the best obtainable rate, which is the linear minimax rate $l_n$ \mods{defined in \eqref{eq:lminimax}}. As can be seen at the end of the proof of Proposition \ref{prop:rescaledrate}, this rate is achieved for any $\alpha\geq\beta-\frac{1}q$ for appropriate rescaling (or for no rescaling, when $\alpha=\beta-\frac{\gamma}2$, in agreement with Remark \ref{rem:l2ratesqle2}(ii)). {We reiterate, that at this stage it is not clear whether Gaussian priors are fundamentally limited by the linear minimax rate,} {see the discussion in Section H of the supplementary material below.}
\item[iv)] The case $w_0\in B^1_1$, which is particularly interesting in applications like signal processing, is not immediately covered by Theorem \ref{thm:wn3}. However, an inspection of the proof of Proposition \ref{prop:rescaledrate} shows that using a rescaled $\alpha$-regular $p$-exponential prior, with $p=1$ and for $\alpha>0$ arbitrarily small, we can achieve a rate of contraction arbitrarily close to the minimax rate \mods{$m_n$ in \eqref{eq:nlminimax}}.
\end{enumerate}
\end{remark}

\mods{\begin{remark}
The product Laplace prior ($p=1$ and $\gamma_\ell=1, \,\forall \ell\in\N$), has been studied in the sparse Gaussian sequence model in \cite[Section 3]{CSV15}. Even though the posterior mode corresponds to the LASSO, which is known to provide minimax optimal estimation in this setting, \cite[Theorem 7]{CSV15} shows that the whole posterior contracts at a suboptimal rate for truly sparse signals. This is because the posterior variance is overly large. On the other hand, Theorem \ref{thm:wn3} shows that in the white noise model, under sparsity assumptions expressed in terms of Besov regularity with integrability parameter $q<2$, the appropriately tuned rescaled $\alpha$-regular Laplace prior results in a posterior which contracts optimally for $q=1$ or optimally up to logarithmic terms for $q\neq 1$.  
\end{remark}}

\subsection{Estimates relating to the concentration function}\label{sec:conc}

\subsubsection{$\alpha$-regular $p$-exponential priors in $\ell_2$}
We first study  the centered small ball probability term in the concentration function.  The result is a direct consequence of \cite[Theorem 4.2]{FA07}.
\begin{lemma}\label{lem:l2sb}
Assume $\pe$ is a $d$-dimensional $\alpha$-regular $p$-exponential measure in $\ell_2$. Then as $\epsilon\to0$
\[-\log\pe(\epsilon B_{\ell_2})\asymp\epsilon^{-\frac{d}\alpha}.\]
\end{lemma}

In the next lemma we compute upper bounds on the first term in the concentration function $\varphi_w$, depending on the Besov regularity of a $w\in\ell_2$. 

\begin{lemma}\label{lem:infl2}
Assume that $\pe$ is a $d$-dimensional $\alpha$-regular $p$-exponential measure in $\ell_2$ and that $w_0\in B^\beta_q$ for $\beta>0\vee(\frac{d}q-\frac{d}2)$, $q\geq1$.
Then as $\epsilon\to0,$ we have the following bounds where all the constants depend on $w_0$ only through its $B^\beta_q$-norm: 
\begin{enumerate}
\item[i)]For $q\leq p$ (hence $q\leq2$)
\[\inf_{h\in \oma: \norm{h-w_0}_{\ell_2}\leq \epsilon}\norm{h}^p_{\oma} \lesssim 
	\begin{cases}	
		1,& {\rm if }\; \beta \geq \alpha + \frac dq, \; \\
		\epsilon^{2p\frac{(\beta- \alpha)q-d}{(2\beta+d)q-2d}}, & {\rm if }\; \beta < \alpha + \frac dq. \; \\		
\end{cases}\]
\item[ii)]For $q>p$
\[\inf_{h\in \oma: \norm{h-w_0}_{\ell_2}\leq \epsilon}\norm{h}^p_{\oma} \lesssim 
	\begin{cases}	
		1, & {\rm if }\; \beta > \alpha + \frac dp, \; \\
		(-\log \epsilon)^{\frac{q-p}q}, & {\rm if }\; \beta = \alpha + \frac dp, \;  \\
		\epsilon^{\frac{\beta p-\alpha p-d}{\beta}},  & {\rm if }\; \beta < \alpha + \frac dp \,\;{\rm and}\,\;q\geq2, \; \\
		\epsilon^{2q\frac{(\beta-\alpha) p-d}{(2\beta+d)q-2d}},  & {\rm if }\; \beta < \alpha + \frac dp \;\,{\rm and}\,\;q<2. 
\end{cases}\]
\end{enumerate}
\end{lemma}

We finally verify that for $\epsilon_n$ the rate in Proposition \ref{prop:ubl2}, the quality of approximation of $\sh=\sha$ by $\om=\oma$ in $\ell_2$, is sufficiently good for  the maximum appearing in the right hand side of the complexity bound \eqref{eq:adh22} in our general contraction Theorem \ref{t:main}, to be dominated by $n\epsilon_n^2$. 
\begin{lemma}\label{lem:l2dom}
Let $\pe$ be a $d$-dimensional $\alpha$-regular $p$-exponential measure in $\ell_2$. 
For $\epsilon, a>0$, define
\begin{equation}
\label{eq:def_f_and_g}
f(a)=a^p(1\vee a^{\frac{2d-pd}{d+2\alpha}})\,\;\text{and}\;\; g(\epsilon) = 2(1\vee\epsilon^{-\frac{2d}{d+2\alpha}}).
\end{equation}
Then $f$ and $g$ satisfy the approximation bound \eqref{eq:inf} in Section \ref{sec:main} and, moreover, 
\begin{equation}
	\label{eq:g_ineq}
	f(n^\frac12\epsilon_n)g(\epsilon_n)^{1-\frac{p}2}\lesssim n \epsilon_n^2 
\end{equation}
for all $p\in[1,2]$, where $\epsilon_n=r_n^{\alpha,\beta,p,q}$ for $r_n^{\alpha, \beta, p ,q}$ as in Proposition \ref{prop:ubl2}. 
\end{lemma}

\subsubsection{Rescaled $\alpha$-regular $p$-exponential priors in $\ell_2$}
The expression \eqref{eq:rescf} for the concentration function $\bar\varphi_w(\cdot)$ of rescaled $\alpha$-regular $p$-exponential priors, suggests that we can re-use Lemmas \ref{lem:l2sb} and \ref{lem:infl2} in order to prove Proposition~\ref{prop:rescaledrate}.  
In the next lemma we also verify that for $\epsilon_n$ the rate in Proposition \ref{prop:rescaledrate}, the quality of approximation of $\sh=\shl$ by $\om=\oml$ in $\ell_2$, is sufficiently good for  the maximum appearing in the right hand side of the complexity bound \eqref{eq:adh22} in our general contraction Theorem \ref{t:main}, to be dominated by $n\epsilon_n^2$. 
\begin{lemma}\label{lem:l2domresc}
Let $\pel$ be a rescaled $d$-dimensional $\alpha$-regular $p$-exponential measure in $\ell_2$. 
For $t>0$, define $\f(t):=\lambda^{-p}f(\lambda t)$ and $\g(t):=g(t)$, where $f, g$ are as in Lemma~\ref{lem:l2dom}.
Then $\f$ and $\g$ satisfy the approximation bound \eqref{eq:inf} in Section \ref{sec:main}, and, moreover, 
\begin{equation}
	\label{eq:g_ineql}
	\f(n^\frac12\epsilon_n)\g(\epsilon_n)^{1-\frac{p}2}\lesssim n \epsilon_n^2 
\end{equation}
for all $p\in[1,2]$, where $\epsilon_n=\bar{r}_n$ and $\lambda=\lambda_n$ as in Proposition \ref{prop:rescaledrate}. 
\end{lemma}

\section{The $C[0,1]$ setting}\label{sec:l8}
In this section we consider $p$-exponential measures in the separable Banach space $X=C[0,1].$ We define $p$-exponential measures using an appropriately regular  Schauder basis, \mods{see below for details}. In particular, we define $\alpha$-regular $p$-exponential measures in $C[0,1]$ and study their concentration at centers of varying H\"older-type regularity. We combine our findings with Theorem \ref{thm:de}, to obtain posterior contraction rates for density estimation under H\"older-type regularity of the truth.

We consider orthonormal wavelet bases of $L_2[0,1]$, constructed as discussed in \cite{CDV93}; see \cite{GN16} or \cite{Meyer92} for fundamentals of wavelet analysis. We denote such a wavelet basis by $\{\psi_{kl}: k\in\N_0, l=1,\dots, 2^k\}$, where $k$ corresponds to the resolution level and $l$ to the location. A function $\uf\in L_2[0,1]$ can be expanded as $\sum_{k=0}^\infty\sum_{l=1}^{2^k}\uf_{kl}\psi_{kl}$, where the coefficients $\uf_{kl}$ are given by the $L_2$-inner products between $\uf$ and $\psi_{kl}$. We assume that $\psi_{kl}$ are $S$-H\"older continuous for some $S>0$. We record some properties that will be useful for our analysis, see \cite{CDV93, GN16}:
\begin{itemize}
\item $\{\psi_{kl}\}$ is a Schauder basis of $C[0,1]$. 
\item There exists a constant $C_1>0$ such that 
\begin{equation}\label{eq:hold}
|\psi_{kl}(x)-\psi_{kl}(y)|\leq C_1 2^{\frac{k}2+k\vartheta}|x-y|^\vartheta, \;\vartheta\leq S\wedge1.
\end{equation}
\item There exists a constant $C_2>0$ such that
\begin{equation}\label{eq:supnorm}
\Bigg\|\sum_{l=1}^{2^k}\uf_{kl} \psi_{kl}\Bigg\|_{L_\infty}\leq  C_2\,2^{\frac{k}2}\sup_{1\leq l\leq 2^k}|\uf_{kl}|.
\end{equation}
\item Let $0<s<S$. Then $\uf$ belongs to the Besov space $B^s_{\infty\infty}[0,1]$ if and only if  
$$
\norm{\uf}_{B^s_{\infty\infty}}:=\sup_{k\ge 0;\, 1\le l\le 2^k}2^{k(\frac{1}{2}+s)}|\uf_{kl} |<\infty. 
$$
Furthermore, if $s$ is non-integer we have that $g\in C^s$ if and only if \[\norm{\uf}_{B^s_{\infty\infty}}<\infty.\] \end{itemize}
Note, that our analysis holds for other possibly nonorthonormal multiresolution Schauder bases, provided the above bounds on $\psi_{kl}$ and the characterizations in terms of the coefficients $\uf_{kl}$ hold. For example, one can use {the Faber (integrated Haar) basis, see \cite[Section 3.1.3]{TR10}}. We use basis functions $\psi_{kl}$ which have sufficient H\"older regularity, so that $\psi_{kl}$ can characterize the maximal $(s,\infty,\infty)$-Besov (or $s$-H\"older) regularity we consider, that is we assume $S>\max\{\alpha,\beta\}$, where $\alpha,\beta$ will express the regularity of the prior and truth, respectively.

We can define a $p$-exponential measure $\pe$ in $C[0,1]$ by randomizing the coefficients in the expansion
$$
\uf(t)=\sum_{k=0}^\infty \sum_{l=1}^{2^k}\uf_{kl}\psi_{kl}(t), \quad t\in[0,1].
$$
We let  \begin{equation}\label{eq:randomtent}\uf_{kl}=\gamma_{kl}\xi_{kl}, \quad \xi_{kl}\stackrel{iid}{\sim}f_p,\quad p\in[1,2],  \quad\gamma_{kl}= 2^{-(\frac12+\alpha)k}, \;\alpha>0. \end{equation}   

 The next result studies H\"older continuity of draws from this $p$-exponential measure.

\begin{proposition}\label{prop:hol}
Let $\pe$ be the $p$-exponential measure defined in \eqref{eq:randomtent}, for any $p\in[1,2]$ and $\alpha>0$. Then $\pe(C^s)=1$ for all  $s<\alpha\wedge1$. \footnote{\mods{By assuming additional regularity on the basis functions and using techniques relying on the embeddings of Besov spaces, it is possible to remove the requirement $s<1$ in Proposition \ref{prop:hol}. See \cite[Theorem 7]{DS15} for a similar derivation and \cite[Section 4.6.1]{HT78} for the relevant embeddings.}}
\end{proposition}
In particular, the last proposition implies that indeed $\pe$ is a measure on $X=C[0,1]$. We call $\pe$ defined in \eqref{eq:randomtent} an  \emph{$\alpha$-regular $p$-exponential measure in $C[0,1]$.}

By Proposition \ref{prop:supp2} the topological support of $\pe$ is the space $C[0,1]$.
We next study the concentration function $\varphi_w(\cdot)$ of $\pe$, defined for centers $w\in C[0,1]$; see Definition \ref{defn:conc} where $X=C[0,1]$, with $\norm{\cdot}_X=\norm{\cdot}_{L_\infty}$. The next lemma identifies the space $\om$ in which we approximate the center $w\in C[0,1]$ in the first term of $\varphi_w$, as well as the shift space $\sh$. Note that these spaces can be defined in sequence space, independently of the parameter space $X$ and the Schauder basis in which we work. The lemma follows immediately by Proposition \ref{p:shift} and Definition \ref{defn:om}. 

\begin{lemma}\label{lem:omshl8}
Assume $\pe$ is an $\alpha$-regular $p$-exponential measure in $C[0,1]$. Then 
\[\om=\oma:=\{h\in\R^\infty: \sum_{k=1}^\infty\sum_{l=1}^{2^k} |h_{kl}|^p2^{(\frac{1}2+\alpha)pk}<\infty\}, \;\norm{h}_{\oma}=\Big(\sum_{k=1}^\infty\sum_{l=1}^{2^k} |h_{kl}|^p2^{(\frac{1}2+\alpha)pk}\Big)^\frac1p,\] 
\[\sh=\sha:=\{h\in\R^\infty: \sum_{k=1}^\infty\sum_{l=1}^{2^k}h_{kl}^22^{(1+2\alpha)k}<\infty\}, \;\norm{h}_{\sha}=\Big(\sum_{k=1}^\infty\sum_{l=1}^{2^k} h_{kl}^22^{(1+2\alpha)k }\Big)^\frac12.\] 
\end{lemma}
In fact, due to the asymptotic equivalence of the sequences $\gamma_{kl}=2^{-(\frac12+\alpha)k}$, $k\in\N, 1\leq l\leq 2^k$ and $\gamma_\ell=\ell^{-\frac12-\alpha}, \ell\in\N,$ we have that $\oma=B^{\alpha+\frac1p}_p$ and $\sha=B^{\alpha+\frac12}_2$, where $B^s_q$ are the Besov-type spaces of sequences defined in Definition \ref{defn:besov} for $d=1$.

We next study the centered small ball probability term in the concentration function. For the proof we use the techniques of \cite{WS96} which studies the Gaussian case.

\begin{proposition}\label{prop:supnorm}
Let $\pe$ be an $\alpha$-regular $p$-exponential measure in $C[0,1].$ Then as $\epsilon\to0$
\[-\log\pe(\{u\in C[0,1]: \norm{u}_{L_\infty}\leq \epsilon\})\lesssim \epsilon^{-\frac1\alpha}.\]
\end{proposition}
Finally, in the next lemma we compute upper bounds on the first term in the concentration function $\varphi_w$, depending on the $(\beta,\infty,\infty)$-Besov regularity of $w$, which recall is identified with $\beta$-H\"older regularity when $\beta$ is non-integer. 
\begin{lemma}\label{lem:infl8}
Assume that $\pe$ is an $\alpha$-regular $p$-exponential measure in $C[0,1]$ and that $w_0\in B^\beta_{\infty\infty}$, $\beta>0$. Then, as $\epsilon\to0$
\begin{equation*}\inf_{h\in\oma:\norm{h-{w_0}}_{L_\infty}\leq\epsilon}\norm{h}_{\oma}^p\lesssim\left\{\begin{array}{ll}
    \epsilon^{\frac{\beta p -\alpha p -1}\beta}, & \;\mbox{if \, $\beta<\alpha+\frac1p$,}

                                     \\ \log(1/\epsilon), & \;\mbox{if \, $\beta=\alpha+\frac1p$,}
\\ 1, & 
\;\mbox{if \, $\beta> \alpha+\frac1p$.} 
                                                                    \end{array}\right.\end{equation*}
\end{lemma}
Combining the previous lemmas, we can find upper bounds on the minimal solution $\epsilon_n$ of the inequality $\varphi_{w_0}(\epsilon_n)\leq n\epsilon_n^2$ depending on the H\"older regularity of $w_0$. Since the rates on the right hand sides of the bounds in Proposition \ref{prop:supnorm} and Lemma \ref{lem:infl8} are identical to the ones in {Lemmas \ref{lem:l2sb} and \ref{lem:infl2} (for $d=1$ and $q\geq2$), respectively, the proof is identical to the proof of part (i) of Proposition \ref{prop:ubl2} and is hence omitted.}
\begin{proposition}\label{prop:ubl8}
Assume that $\pe$ is an $\alpha$-regular $p$-exponential measure in $C[0,1]$ and that $w_0\in B^\beta_{\infty\infty}$, $\beta>0$. Then
as $n\to\infty$ the rate $\epsilon_n\asymp\rho_n^{\alpha,\beta,p}$ satisfies the inequality $\varphi_{w_0}(\epsilon_n)\leq n\epsilon_n^2$, where \[\rho_n^{\alpha,\beta,p}:=\left\{\begin{array}{ll}n^{-\frac{\beta}{1+2\beta+p(\alpha-\beta)}}, &  
\;\mbox{if \,$\beta\leq \alpha$,} 
                                     \\ n^{-\frac{\alpha}{1+2\alpha}}, &  \;\mbox{if \,$\beta>\alpha$.}
                                    \end{array}\right.\]
\end{proposition}

We next study the quality of the approximation of elements of $\sha$ by elements of $\oma$ in the supremum norm, that is, determine functions $g, f$ such that \eqref{eq:inf} in Section \ref{sec:main} holds. 

\begin{lemma}\label{lem:l8dom}
Let $\pe$ be an $\alpha$-regular $p$-exponential measure in $C[0,1]$. Then there exists $c>0$ depending only on the Schauder basis, $p$ and $\alpha$, such that the functions $f(a)=ca^{\frac{2-p+2\alpha p}{2\alpha}}$ and $g(\epsilon)=\epsilon^{-\frac1\alpha}$ satisfy \eqref{eq:inf}. 
\end{lemma}

A straightforward computation shows that for the above $g$ and $f$, the rate $\epsilon_n=\rho_n^{\alpha,\beta,p}$ is such that the right hand side of the complexity bound \eqref{eq:adh22} in Theorem \ref{t:main} is dominated by $n\epsilon_n^2$ only if $\beta\leq \alpha-\frac1{2-p+2\alpha p}$. This means that the complexity bound we obtain from Theorem \ref{t:main} does not match the conditions of general results like \cite[Theorems 8.9 and 8.19]{GV17} and we need to use Corollary \ref{c:main} to get contraction rates. To this end we solve \eqref{eq:sub} and find that for these functions $f,g$ and for $\epsilon_n=\rho_n^{\alpha,\beta,p}$, the fastest decaying solution is $\tilde{\epsilon}_n\asymp\tilde{\rho}_n^{\alpha,\beta,p},$ where
\begin{equation}\label{eq:ratel8}\tilde{\rho}_n^{\alpha,\beta,p}:=\left\{\begin{array}{ll}n^{\frac{(2-p)(1-2\alpha)}{8\alpha}-\frac{p\beta}{2(1+2\beta+p(\alpha-\beta))}}, &  
\;\mbox{if \,$\beta\leq \alpha$,} 
                                     \\ n^{\frac{2-p-8\alpha^2}{8\alpha(1+2\alpha)}}, &  \;\mbox{if \,$\beta>\alpha$.}
                                    \end{array}\right.\end{equation}
For a fixed value of the regularity of the truth $\beta>0$, note that  $\tilde{\rho}_n^{\alpha,\beta,p}$ decays only for sufficiently large prior regularity $\alpha$. For example, if $\alpha<\beta$, we have decay only for  $\alpha>\sqrt{\frac{2-p}8}$. As $p\to2$, since the difficulty in the complexity bound \eqref{eq:adh22} disappears, the rates $\tilde{\rho}_n^{\alpha,\beta,p}$ approach the rates $\rho_n^{\alpha,\beta,p}$.

For example, combining these considerations with Theorem \ref{thm:de}, we get immediately the following result giving contraction rates for density estimation.

\begin{theorem}\label{thm:decontr}
Consider the density estimation model  of Subsection \ref{ssec:de}, and let $W$ be an $\alpha$-regular $p$-exponential random element in $C[0,1]$, $\alpha>0, p\in[1,2]$. Assume $w_0=\log\pi_0\in B^\beta_{\infty\infty}, \beta>0$ and denote by $P^n_0$ the distribution of the vector $(X_1,\dots,X_n)$. Let $\rho_n^{\alpha,\beta,p}, \tilde{\rho}_n^{\alpha,\beta,p}$ be defined as in Proposition \ref{prop:ubl8} and \eqref{eq:ratel8}, respectively. Then for $M$ large enough, as $n\to\infty$
\[\Pi_n(\pi:d_H(\pi,\pi_0)>M(\rho_n^{\alpha,\beta,p}\vee\tilde{\rho}_n^{\alpha,\beta,p})|X_1,\dots,X_n)\to0,\] in $P^n_0$-probability. 
\end{theorem}

For $p=2$, we have that $\rho_n^{\alpha,\beta,2}=\tilde{\rho}_n^{\alpha,\beta,2}$ and we recover existing contraction rates for Gaussian  priors, see \cite[Section 11.4]{GV17} or  \cite[Theorem 7.3.9]{GN16}. In this case, if the regularity of the prior matches the regularity of the truth $\alpha=\beta$, we get the minimax estimation rate in the Hellinger distance for functions which are $\beta$-H\"older continuous, $n^{-\frac{\beta}{1+2\beta}}$.  \mods{For $p\in[1,2)$, a straightforward calculation shows that the rate $\tilde{\rho}_n^{\alpha,\beta,p}$ is slower than $\rho_n^{\alpha,\beta,p}$ unless $0<\beta\leq \alpha-\frac1{2-p+2\alpha p}$. In particular, for $\alpha=\beta$ we only have contraction if the prior is sufficiently regular, $\alpha>\sqrt{\frac{2-p}8}$, with contraction rate $\tilde{\rho}_n^{\alpha,\alpha,p}$ which is slower than the minimax rate.} As $p$ increases towards $p=2$ the gap disappears; likewise for large $\alpha$.

It appears that contrary to the Gaussian case $p=2$, studying prior concentration and using general contraction results relying on prior mass and entropy conditions, is not optimal for proving contraction rates for $p$-exponential priors in $C[0,1]$ when $p\in[1,2)$. This is due to the more complicated complexity bound \eqref{eq:adh22} compared to the Gaussian case, which in this setting affects the rates because of the poor approximation quality of $\sha$ by $\oma$ in the supremum norm. {Note that rescaling the prior as considered in Subsection \ref{ssec:resc}, does not help with this issue}. In general contraction results like \cite[Theorem 8.9]{GV17}, the entropy condition is used to construct certain necessary tests. 
Directly constructing the necessary tests and using other general contraction results which do not rely on entropy conditions, for example \cite[Theorem 8.12]{GV17}, may resolve this issue. This is out of the scope of the present paper, but it is a possible future direction.

{\bf Acknowledgements} {The authors would like to thank Emanuel Milman and Mikhail Lifshits for providing very useful guidance regarding isoperimetry and small ball probabilities, respectively. SA would also like to thank Richard Nickl, Ismael Castillo, Kolyan Ray and Aad van der Vaart for answering many questions and providing important insight regarding posterior contraction rates. SA is particularly grateful to Sven Wang for sharing the concept of using rescaling for improving rates rather than just for adaptation, and to Richard Nickl for providing a very stimulating research environment during a research visit to Cambridge. Finally, the authors are thankful to an anonymous associate editor and two anonymous referees, for corrections and suggestions that led to a substantial improvement of some of the results in this article.}

\bibliographystyle{abbrv}
\bibliography{references}

\begin{thebibliography}{10}

\bibitem{ABDH18}
S.~Agapiou, M.~Burger, M.~Dashti, and T.~Helin.
\newblock Sparsity-promoting and edge-preserving maximum a posteriori
  estimators in non-parametric {B}ayesian inverse problems.
\newblock {\em Inverse Problems}, 34(4):045002, 2018.

\bibitem{APSS17}
S.~Agapiou, O.~Papaspiliopoulos, D.~Sanz-Alonso, and A.~M. Stuart.
\newblock Importance sampling: Intrinsic dimension and computational cost.
\newblock {\em Statistical Science}, 32(3):405--431, 2017.

\bibitem{AGR13}
J.~Arbel, G.~Gayraud, and J.~Rousseau.
\newblock Bayesian optimal adaptive estimation using a sieve prior.
\newblock {\em Scandinavian journal of statistics}, 40(3):549--570, 2013.

\bibitem{FA07}
F.~Aurzada.
\newblock On the lower tail probabilities of some random sequences in lp.
\newblock {\em Journal of Theoretical Probability}, 20(4):843--858, 2007.

\bibitem{BG03}
E.~Belitser and S.~Ghosal.
\newblock Adaptive bayesian inference on the mean of an infinite-dimensional
  normal distribution.
\newblock {\em The Annals of Statistics}, 31(2):536--559, 2003.

\bibitem{Boga_gaussian}
V.~I. Bogachev.
\newblock {\em Gaussian measures}, volume~62 of {\em Mathematical Surveys and
  Monographs}.
\newblock American Mathematical Society, Providence, RI, 1998.

\bibitem{Boga10}
V.~I. Bogachev.
\newblock {\em Differentiable measures and the {M}alliavin calculus}, volume
  164 of {\em Mathematical Surveys and Monographs}.
\newblock American Mathematical Society, Providence, RI, 2010.

\bibitem{CB74}
C.~Borell.
\newblock Convex measures on locally convex spaces.
\newblock {\em Arkiv f{\"o}r Matematik}, 12(1):239--252, 1974.

\bibitem{CB75}
C.~Borell.
\newblock The brunn-minkowski inequality in gauss space.
\newblock {\em Inventiones mathematicae}, 30(2):207--216, 1975.

\bibitem{rev5}
E.~Candes and J.~Romberg.
\newblock l1-magic: Recovery of sparse signals via convex programming.
\newblock {\em URL: www.acm.caltech.edu/l1magic/downloads/l1magic. pdf}, 4:14,
  2005.

\bibitem{rev6}
E.~J. Candes and D.~L. Donoho.
\newblock Curvelets: A surprisingly effective nonadaptive representation for
  objects with edges.
\newblock Technical report, Stanford University, Department of Statistics,
  2000.

\bibitem{IC08}
I.~Castillo.
\newblock Lower bounds for posterior rates with gaussian process priors.
\newblock {\em Electronic Journal of Statistics}, 2:1281--1299, 2008.

\bibitem{CN13}
I.~Castillo and R.~Nickl.
\newblock Nonparametric {B}ernstein--von {M}ises theorems in {G}aussian white
  noise.
\newblock {\em The Annals of Statistics}, 41(4):1999--2028, 2013.

\bibitem{CSV15}
I.~Castillo, J.~Schmidt-Hieber, and A.~van~der Vaart.
\newblock Bayesian linear regression with sparse priors.
\newblock {\em Ann. Statist.}, 43(5):1986--2018, 2015.

\bibitem{CDV93}
A.~Cohen, I.~Daubechies, and P.~Vial.
\newblock Wavelets on the interval and fast wavelet transforms.
\newblock {\em Applied and computational harmonic analysis}, 1993.

\bibitem{GDP06}
G.~Da~Prato.
\newblock {\em An introduction to infinite-dimensional analysis}.
\newblock Springer Science \&amp; Business Media, 2006.

\bibitem{DS15}
M.~Dashti and A.~M. Stuart.
\newblock The {B}ayesian approach to inverse problems.
\newblock In {\em Handbook of uncertainty quantification. {V}ol. 1, 2, 3},
  pages 311--428. Springer, Cham, 2017.

\bibitem{rev7}
I.~Daubechies, M.~Defrise, and C.~De~Mol.
\newblock An iterative thresholding algorithm for linear inverse problems with
  a sparsity constraint.
\newblock {\em Communications on pure and applied mathematics},
  57(11):1413--1457, 2004.

\bibitem{rev8}
D.~L. Donoho and M.~Elad.
\newblock Optimally sparse representation in general (nonorthogonal)
  dictionaries via l1 minimization.
\newblock {\em Proceedings of the National Academy of Sciences},
  100(5):2197--2202, 2003.

\bibitem{rev9}
D.~L. Donoho and X.~Huo.
\newblock Uncertainty principles and ideal atomic decomposition.
\newblock {\em IEEE Transactions on Information Theory}, 47(7):2845--2862,
  2001.

\bibitem{DJ98}
D.~L. Donoho and I.~M. Johnstone.
\newblock Minimax estimation via wavelet shrinkage.
\newblock {\em The Annals of Statistics}, 26(3):879--921, 1998.

\bibitem{DLM90}
D.~L. Donoho, R.~C. Liu, and B.~MacGibbon.
\newblock Minimax risk over hyperrectangles, and implications.
\newblock {\em Ann. Statist.}, 18(3):1416--1437, 1990.

\bibitem{GGV00}
S.~Ghosal, J.~K. Ghosh, and A.~W. Van Der~Vaart.
\newblock Convergence rates of posterior distributions.
\newblock {\em Annals of Statistics}, pages 500--531, 2000.

\bibitem{GV07}
S.~Ghosal and A.~Van Der~Vaart.
\newblock Convergence rates of posterior distributions for noniid observations.
\newblock {\em The Annals of Statistics}, 35(1):192--223, 2007.

\bibitem{GV17}
S.~Ghosal and A.~van~der Vaart.
\newblock {\em Fundamentals of nonparametric {B}ayesian inference}.
\newblock Cambridge Series in Statistical and Probabilistic Mathematics.
  Cambridge University Press, 2017.

\bibitem{GN11}
E.~Gin{\'e} and R.~Nickl.
\newblock Rates of contraction for posterior distributions in lr-metrics.
\newblock {\em The Annals of Statistics}, 39(6):2883--2911, 2011.

\bibitem{GN16}
E.~Gin{\'e} and R.~Nickl.
\newblock {\em Mathematical foundations of infinite-dimensional statistical
  models}, volume~40.
\newblock Cambridge University Press, 2016.

\bibitem{HKPT98}
W.~H\"{a}rdle, G.~Kerkyacharian, D.~Picard, and A.~Tsybakov.
\newblock {\em Wavelets, approximation, and statistical applications}, volume
  129 of {\em Lecture Notes in Statistics}.
\newblock Springer-Verlag, New York, 1998.

\bibitem{HB15}
T.~Helin and M.~Burger.
\newblock Maximum a posteriori probability estimates in infinite-dimensional
  {B}ayesian inverse problems.
\newblock {\em Inverse Problems}, 31(8):085009, 2015.

\bibitem{rev10}
I.~M. Johnstone.
\newblock Minimax bayes, asymptotic minimax and sparse wavelet priors.
\newblock In {\em Statistical Decision Theory and Related Topics V}, pages
  303--326. Springer, 1994.

\bibitem{AK03}
A.~Kagan.
\newblock Statistical approach to some mathematical problems.
\newblock {\em Austrian Journal of Statistics}, 32:71--83, 2003.

\bibitem{OV06}
O.~Kallenberg.
\newblock {\em Foundations of modern probability}.
\newblock Springer Science \&amp; Business Media, 2006.

\bibitem{KW70}
G.~S. Kimeldorf and G.~Wahba.
\newblock A correspondence between bayesian estimation on stochastic processes
  and smoothing by splines.
\newblock {\em The Annals of Mathematical Statistics}, 41(2):495--502, 1970.

\bibitem{KSVZ12}
B.~T. Knapik, B.~T. Szab\'{o}, A.~W. van~der Vaart, and J.~H. van Zanten.
\newblock Bayes procedures for adaptive inference in inverse problems for the
  white noise model.
\newblock {\em Probab. Theory Related Fields}, 164(3-4):771--813, 2016.

\bibitem{KVZ12}
B.~T. Knapik, A.~W. van Der~Vaart, and J.~H. van Zanten.
\newblock {B}ayesian inverse problems with {G}aussian priors.
\newblock {\em The Annals of Statistics}, 39(5):2626--2657, 2011.

\bibitem{KLNS12}
V.~Kolehmainen, M.~Lassas, K.~Niinim{\"a}ki, and S.~Siltanen.
\newblock Sparsity-promoting bayesian inversion.
\newblock {\em Inverse Problems}, 28(2):025005, 2012.

\bibitem{LSS09}
M.~Lassas, E.~Saksman, and S.~Siltanen.
\newblock Discretization-invariant {B}ayesian inversion and {B}esov space
  priors.
\newblock {\em Inverse Probl. Imaging}, 3(1):87--122, 2009.

\bibitem{ML05}
M.~Ledoux.
\newblock {\em The concentration of measure phenomenon}.
\newblock Number~89. American Mathematical Soc., 2005.

\bibitem{PL88}
P.~J. Lenk.
\newblock The logistic normal distribution for bayesian, nonparametric,
  predictive densities.
\newblock {\em Journal of the American Statistical Association},
  83(402):509--516, 1988.

\bibitem{MP03}
P.~Math\'{e} and S.~V. Pereverzev.
\newblock Geometry of linear ill-posed problems in variable {H}ilbert scales.
\newblock {\em Inverse Problems}, 19(3):789--803, 2003.

\bibitem{Meyer92}
Y.~Meyer.
\newblock Wavelets and operators.
\newblock volume 37 of Cambridge Studies in Advanced Mathematics, 1992.

\bibitem{PPRS12}
O.~Papaspiliopoulos, Y.~Pokern, G.~O. Roberts, and A.~M. Stuart.
\newblock Nonparametric estimation of diffusions: a differential equations
  approach.
\newblock {\em Biometrika}, 99(3):511--531, 2012.

\bibitem{PSZ13}
Y.~Pokern, A.~M. Stuart, and J.~H. van Zanten.
\newblock Posterior consistency via precision operators for {B}ayesian
  nonparametric drift estimation in {SDE}s.
\newblock {\em Stochastic Process. Appl.}, 123(2):603--628, 2013.

\bibitem{KR13}
K.~Ray et~al.
\newblock Bayesian inverse problems with non-conjugate priors.
\newblock {\em Electronic Journal of Statistics}, 7:2516--2549, 2013.

\bibitem{CR10}
C.~Roberto.
\newblock Isoperimetry for product of probability measures: recent results.
\newblock {\em Markov Process. Related Fields}, 16(4):617--634, 2010.

\bibitem{RS17}
J.~Rousseau and B.~Szabo.
\newblock Asymptotic behaviour of the empirical {B}ayes posteriors associated
  to maximum marginal likelihood estimator.
\newblock {\em Ann. Statist.}, 45(2):833--865, 2017.

\bibitem{schmeisser1987topics}
H.~J. Schmeisser and H.~Triebel.
\newblock {\em Topics in {F}ourier analysis and function spaces}, volume~42 of
  {\em Mathematics and its Applications in Physics and Technology}.
\newblock Akademische Verlagsgesellschaft Geest \& Portig K. G., Leipzig, 1987.

\bibitem{SW01}
X.~Shen and L.~Wasserman.
\newblock Rates of convergence of posterior distributions.
\newblock {\em The Annals of Statistics}, 29(3):687--714, 2001.

\bibitem{LAS65}
L.~Shepp.
\newblock Distinguishing a sequence of random variables from a translate of
  itself.
\newblock {\em The Annals of Mathematical Statistics}, 36(4):1107--1112, 1965.

\bibitem{SKW99}
T.~S. Shively, R.~Kohn, and S.~Wood.
\newblock Variable selection and function estimation in additive nonparametric
  regression using a data-based prior.
\newblock {\em Journal of the American Statistical Association},
  94(447):777--794, 1999.

\bibitem{WS96}
W.~Stolz.
\newblock Some small ball probabilities for {G}aussian processes under
  nonuniform norms.
\newblock {\em Journal of Theoretical Probability}, 9(3):613--630, 1996.

\bibitem{stuart}
A.~M. Stuart.
\newblock Inverse problems: a {B}ayesian perspective.
\newblock {\em Acta Numer.}, 19:451--559, 2010.

\bibitem{SVZ13}
B.~T. Szab{\'o}, A.~W. van~der Vaart, and J.~H. van Zanten.
\newblock Empirical {B}ayes scaling of {G}aussian priors in the white noise
  model.
\newblock {\em Electronic Journal of Statistics}, 7:991--1018, 2013.

\bibitem{TA94}
M.~Talagrand.
\newblock The supremum of some canonical processes.
\newblock {\em American Journal of Mathematics}, 116(2):283--325, 1994.

\bibitem{HT78}
H.~Triebel.
\newblock {\em Interpolation theory, function spaces, differential operators},
  volume~18 of {\em North-Holland Mathematical Library}.
\newblock North-Holland Publishing Co., Amsterdam-New York, 1978.

\bibitem{TR10}
H.~Triebel.
\newblock {\em Bases in function spaces, sampling, discrepancy, numerical
  integration}, volume~11.
\newblock European Mathematical Society, 2010.

\bibitem{AT09}
A.~B. Tsybakov.
\newblock {\em Introduction to nonparametric estimation}.
\newblock Springer Series in Statistics. Springer, New York, 2009.
\newblock Revised and extended from the 2004 French original, Translated by
  Vladimir Zaiats.

\bibitem{VZ07}
A.~W. Van Der~Vaart and J.~H. van Zanten.
\newblock Bayesian inference with rescaled gaussian process priors.
\newblock {\em Electronic Journal of Statistics}, 1:433--448, 2007.

\bibitem{VZ08}
A.~W. van~der Vaart and J.~H. van Zanten.
\newblock Rates of contraction of posterior distributions based on gaussian
  process priors.
\newblock {\em The Annals of Statistics}, pages 1435--1463, 2008.

\bibitem{VZrkhs}
A.~W. van~der Vaart and J.~H. van Zanten.
\newblock Reproducing kernel {H}ilbert spaces of {G}aussian priors.
\newblock pages 200--222, 2008.

\bibitem{VZ09}
A.~W. van~der Vaart and J.~H. van Zanten.
\newblock Adaptive bayesian estimation using a {G}aussian random field with
  inverse gamma bandwidth.
\newblock {\em The Annals of Statistics}, 37(5B):2655--2675, 2009.

\bibitem{LZ00}
L.~H. Zhao.
\newblock Bayesian aspects of some nonparametric problems.
\newblock {\em Annals of statistics}, pages 532--552, 2000.

\end{thebibliography}
\newpage
\vspace{0.4cm}
\begin{center}{\bf SUPPLEMENTARY MATERIAL}\end{center}

\setcounter{section}{0}
\setcounter{equation}{0}
\renewcommand\thesection{\Alph{section}}

\renewcommand{\theequation}{\Alph{section}.\arabic{equation}}

{Note that in order to ease readability, in this supplementary material we use a different type of numbering for sections, results and displayed equations, compared to the main body of the article. In particular, we use letters for sections, and the letter of the section together with number, for results and displayed equations.}

\section{Proofs of Section \ref{sec:props}}

\begin{proof}[Proof of {Proposition \ref{prop:supp2}}]
By {Proposition \ref{prop:comp}} we have $\overline{\om}^{\norm{\cdot}_X}\subset\overline{\sh}^{\norm{\cdot}_X}\subset X$. For any arbitrary $x\in X$ and given $\gep>0$, there exists $N$ such that $x^N=\sum_{\ell=1}^N x_\ell \psi_\ell$ satisfies $\|x^N-x\|_X<\gep$, {where $(\psi_\ell)$ is the Schauder basis in $X$ and $(x_\ell)$ the corresponding coefficients of $x$}. Since clearly $x^N\in \sh\cap\om$, we conclude that $X=\overline{\om}^{\norm{\cdot}_X}=\overline{\sh}^{\norm{\cdot}_X}$. 

Since $\pe$ is a measure on $X$, we have ${\rm supp}(\pe)\subset X$.
On the other hand, the topological support of any Radon measure in $X$ is non-empty and by definition closed in $X$. By Proposition \ref{prop:anderson} we get that $0\in{\rm supp}(\pe)$, thus $\sh\subset {\rm supp}(\pe)$. Taking closures in $X$, we get $X=\overline{\sh}^{\norm{\cdot}_X}\subset {\rm supp}(\pe)$ and thus the claimed result.
\end{proof}

\begin{proof}[Proof of {Proposition \ref{prop:lbd}}]
By {Proposition \ref{p:shift}}, letting $V=\frac{|\cdot|^p}p$, we have 
\begin{align*}
& \pe( \epsilon B_X  +h)
= \int_{\epsilon B_X}\lim_{N\to\infty}e^{\left(\sum_{\ell=1}^N\left(V\left(\frac{u_\ell}{\gamma_\ell}\right)-V\left(\frac{u_\ell-h_\ell}{\gamma_\ell}\right)\right)\right)}\pe(du)\\
&=e^{-\sum_{\ell=1}^\infty V\left(\frac{h_\ell}{\gamma_\ell}\right)}
\int_{\epsilon B_X}\lim_{N\to\infty}e^{\sum_{\ell=1}^N\left(V\left(\frac{u_\ell}{\gamma_\ell}\right)+V\left(\frac{h_\ell}{\gamma_\ell}\right)-V\left(\frac{u_\ell-h_\ell}{\gamma_\ell}\right)\right)}\pe(du)\\
&=e^{-\sum_{\ell=1}^\infty V\left(\frac{h_\ell}{\gamma_\ell}\right)}
\int_{\epsilon B_X}\lim_{N\to\infty}\frac{1}{2}\left(e^{\sum_{\ell=1}^N\left(V\left(\frac{u_\ell}{\gamma_\ell}\right)+V\left(\frac{h_\ell}{\gamma_\ell}\right)-V\left(\frac{u_\ell-h_\ell}{\gamma_\ell}\right)\right)}\right.\\
&\hspace{45mm}\left.+e^{\sum_{\ell=1}^N\left(V\left(\frac{u_\ell}{\gamma_\ell}\right)+V\left(\frac{h_\ell}{\gamma_\ell}\right)-V\left(\frac{u_\ell+h_\ell}{\gamma_\ell}\right)\right)}\right)\pe(du),
\end{align*}
where in the last equality we used symmetry.
In the following, we show that the integrand in the last line above is bounded below by $1$. Notice that our proof applies for any $V:\R\to\R_{\geq0}$ convex and symmetric with $V(0)=0$, differentiable in $\R\setminus\{0\}$ with concave derivative on the positive axis. The functions $\frac{|\cdot|^p}p$, $1\leq p \leq 2$, clearly satisfy this assumption.

Since $\e^{a}+\e^{-a}\ge 2$ for any $a\in\R$, we observe that
\begin{align*}
&\e^{V(x)+V(y)-V(x-y)}+\e^{V(x)+V(y)-V(x+y)}\\
&\quad= \e^{V(x)+V(y)-\frac{1}{2}V(x-y)-\frac{1}{2}V(x+y)}
\left(
\e^{-\frac{1}{2}V(x-y)+\frac{1}{2}V(x+y)}+\e^{\frac{1}{2}V(x-y)-\frac{1}{2}V(x+y)}
\right)\\
&\quad\ge 2\,\e^{V(x)+V(y)-\frac{1}{2}V(x-y)-\frac{1}{2}V(x+y)}.
\end{align*}
In consequence, we need to show that
$$
G(x,y):=V(x)+V(y)-\frac{1}{2}V(x-y)-\frac{1}{2}V(x+y)\ge 0.
$$
Notice that $G$ has a number of symmetries. Namely, it satisfies
\begin{equation}
	\label{eq:g_is_symmetric}
	G(x,y) = G(-x,y) = G(x,-y) = G(y,x).
\end{equation}
Therefore, without loss of generality, we can assume that $x,y \ge 0$.  We note that if $x=0$ or $y=0$ then clearly $G(x,y)=V(0) = 0$. 
Consequently, due to \eqref{eq:g_is_symmetric} it will be sufficient to show that $\frac{\partial G}{\partial x}(x,y)\ge 0$ for any $x>0$ and $y>0$.

Let us briefly consider the derivative $R(x)=V'(x)$ for $x>0$ and define $R(0)=\lim_{x\to0^+} V'(x)$. By assumption on $V$, $V'(x)$ is concave hence continuous for all $x>0$, implying that the limit exists although it may be $-\infty$. Combining with the convexity of $V$ and since $V$ has a minimum at the origin, we get that the limit is non-negative, $R(0)\geq0$. The function $R$ defined on $[0,\infty)$ is concave with $R(0)\geq0$, hence it is subadditive. 

We first observe that 
\begin{equation*}
	G(x,x) = 2V(x)-\frac12V(2x)
\end{equation*}
and due to the subadditivity of $R$, we must have $G(x,x)\geq0$. For $x\neq y$, we have
$$
\frac{\partial G}{\partial x}(x,y)=V'(x)-\frac{1}{2}V'(x-y)-\frac{1}{2}V'(x+y).
$$
For $x> y$, by concavity of $V'$ on the positive axis, we have
$$
V'(x)=V'\Big(\frac{1}{2}(x-y)+\frac{1}{2}(x+y)\Big)\ge\frac{1}{2}V'(x-y)+\frac{1}{2}V'(x+y),
$$
implying that $\frac{\partial G}{\partial x}(x,y)\ge 0$ for $x\ge y$. If $x<y$, since by symmetry of $V$  it holds $V'(x-y)=-V'(y-x)$, we can write
$$
\frac{\partial G}{\partial x}(x,y)=V'(x)+\frac{1}{2}V'(y-x)-\frac{1}{2}V'(x+y)
$$
where the arguments of $V'$ in the right-hand side are positive and we can use the concavity of $V'$ on the positive axis. As above, using the auxiliary function $R$ and since concave functions which are non-negative at zero are subadditive, we have 
\[V'(x+y)=V'(2x+y-x)\leq 2V'(x)+V'(y-x).\] 
Thus $\frac{dG}{dx}(x;y)\geq0$ for any $y>x>0$ as well, and the result follows.
\end{proof}

\begin{proof}[Proof of {Theorem \ref{thm:cf}}]
Let $h\in \om$ such that $\norm{h-w}_X\leq \epsilon$. Then by the triangle inequality, for any $x\in X$ we have $\norm{x-w}_X\leq \epsilon +\norm{x-h}_X$, hence if $\norm{x-h}_X
\leq\epsilon$ then $\norm{x-w}_X\leq2\epsilon$. We thus have,
\[\pe(w+2\epsilon B_X)\geq \pe(h+ \epsilon B_X)\geq e^{-\frac1p\norm{h}^p_\om}\pe(\epsilon B_X),\] where for the last inequality we used {Proposition \ref{prop:lbd}}. To finish the proof we take the negative logarithm and optimize over $h\in \om$.
\end{proof}

\begin{proof}[Proof of {Proposition \ref{p:tal}}]
Without loss of generality we work in $\R^\infty$. Recall  $\gamma=(\gamma_\ell)$ and $\xi=(\xi_\ell)$ from the definition of the $p$-exponential measure $\pe$, {Definition \ref{def:pexp}}. The inequality follows from \cite[Theorem 2.4]{TA94}, see also \cite[Theorem 4.19]{ML05}. These theorems state that for the infinite (unscaled) independent product of standard $p$-exponential one-dimensional measures, $\pe_\infty$ in $\R^{\infty}$, there exists a universal constant $K>0$ depending only on $p$, such that for all $\tilde{r}>0$ \[\pe_{\infty}(A+\sqrt{\tilde{r}}B_2+\tilde{r}^\frac1pB_p)\geq1-\frac{1}{\pe_\infty(A)}\exp\left(-\frac{\tilde{r}}K\right),\] where $B_p$, $B_2$ are the closed unit balls in $\ell_p$, $\ell_2$ respectively. Letting $r=\tilde{r}^\frac1p$, we get \[\pe_{\infty}(A+r^\frac{p}2B_2+rB_p)\geq1-\frac{1}{\pe_\infty(A)}\exp\left(-\frac{r^p}K\right).\]
Defining $\Gamma: \R^\infty\to\R^\infty$, such that $x\in\R^\infty \mapsto (\gamma_\ell x_\ell)$, 
we get that for any $\pe$-measurable set $A\subset\R^{\infty}$
\begin{align*}&\pe(A+r^\frac{p}2B_\sh+rB_\om)=\rp(\Gamma \xi\in A+r^\frac{p}2B_\sh+rB_\om)=\rp(\xi\in \Gamma^{-1}(A+r^\frac{p}2B_\sh+rB_\om))\\
&=\pe_\infty( \Gamma^{-1}A+r^\frac{p}2B_2+rB_p)
\geq 1-\frac{1}{\pe_\infty(\Gamma^{-1}A)}\exp\left(-\frac{r^p}K\right)=1-\frac{1}{\pe(A)}\exp\left(-\frac{r^p}K\right).\end{align*}
\end{proof}

\section{Proof of general contraction theorem in Section \ref{sec:main}}
\begin{proof}[Proof of {Theorem \ref{t:main}}] 
Assume $\varphi_{w_0}(\epsilon_n)\leq n\epsilon_n^2$. It follows by  {Theorem \ref{thm:cf}} that
\begin{equation*}
\rp(\norm{W-w_0}_X<2\epsilon_n)
= \exp\left(\log \mu(w_0 + 2 \epsilon_n B_X)\right)
\geq \exp(-\varphi_{w_0}(\epsilon_n))\geq  e^{-n\epsilon_n^2},
\end{equation*}
and, consequently, the {claim \eqref{eq:adh24}} follows.

We now consider the existence of sets $X_n$ such that {\eqref{eq:adh22} and \eqref{eq:adh23}} hold. We set
\begin{equation}
\label{eq:def_Xn}
X_n=\epsilon_nB_X+M_n^{\frac{p}2}B_\sh+M_nB_\om,
\end{equation}
where $M_n>0$ will be chosen below. By {Proposition \ref{p:tal}}, we have 
\begin{equation}
\label{eq:t:main_aux1}
\rp(W\notin X_n) 
\leq \frac{1}{\pe(\epsilon_nB_X)}\exp\left(-\frac{M_n^p}K\right)
= \exp\left(\varphi_0(\epsilon_n)-\frac{M_n^p}K\right).
\end{equation}
Next, for any $C>1$, we denote 
$M_n=(K(C+1)n\epsilon_n^2)^\frac1p$
which is bounded away from zero for all $n$ by assumption. 
Since 
\begin{equation}
\label{eq:t:main_aux2}
\varphi_0(\epsilon_n)\leq \varphi_{w_0}(\epsilon_n)\leq n\epsilon_n^2
\end{equation}
we obtain the {claim \eqref{eq:adh23}} by combining \eqref{eq:t:main_aux1} with \eqref{eq:t:main_aux2}.

\mods{For the final {claim \eqref{eq:adh22}}, we cannot use directly {Proposition \ref{prop:lbd}} to bound the complexity of $X_n$, since {Proposition \ref{prop:lbd}} refers to shifts in $\om$ while $X_n$ involves a ball in $\sh$.  We can however, find a large enough ball $\overline{M}_n\om$ which is such that a $2\epsilon_n$-cushion in $X$ around it contains $X_n$.  We can then use {Proposition \ref{prop:lbd}} to bound the complexity of $2\epsilon_nB_X+ \overline{M}_n B_\om$, which in turn implies a bound on the complexity of $X_n$.}

Define
\begin{equation*}
\overline{M}_n=2\left(M_n\vee (1+\frac1n)f(M_n^{\frac{p}2})^\frac1pg(\epsilon_n)^{\frac1p-\frac12}\right).
\end{equation*}
Then {using \eqref{eq:inf}} we can show that \begin{equation}\label{eq:bigball}X_n\subset 2\epsilon_nB_X+\overline{M}_nB_\om. \end{equation} Indeed, for every $x\in {M_n}^{\frac{p}2}B_\sh$, we have {by \eqref{eq:inf}} 
that 
\begin{equation*}
	\inf_{z\in\om:\norm{z-x}_X\leq\epsilon}\norm{z}_\om^p\leq f(M_n^\frac{p}2)g(\epsilon)^{1-\frac{p}2}
\end{equation*}
and, in consequence, there exists 
$y\in \left(1+\frac1n \right){f(M_n^{\frac{p}2})}^{\frac{1}p}g(\epsilon_n)^{\frac1p-\frac12}B_\om$, 
with $\norm{x-y}_X \leq \epsilon_n$. \mods{The term $1+1/n$ does not play any significant role, and can be replaced by any constant over $1$.}
Thus any $x\in{M_n}^{\frac{p}2}B_\sh+M_nB_\om$  is within $\epsilon_n$ $\norm{\cdot}_X$-distance from some point in $\overline{M}_nB_\om$ and \eqref{eq:bigball} follows. 

Let $h_1,\dots,h_N\in \overline{M}_nB_\om$ be $2\epsilon_n$-apart in $\norm{\cdot}_X$.  Clearly, the balls $h_j+\epsilon_nB_X$ are disjoint and hence by {Proposition \ref{prop:lbd}} we obtain
\begin{multline}\label{pf:adh}
1 \geq \sum_{j=1}^N\rp(W\in h_j+\epsilon_n B_X) 
 \geq \sum_{j=1}^N e^{-\frac{\norm{h_j}^p_{\om}}p} \rp(W\in\epsilon_nB_X)
 \geq  N e^{-\frac{\overline{M}_n^p}p-\varphi_0(\epsilon_n)}.
\end{multline} 
If the set of points $h_1,\dots, h_N$ is maximal in $\overline{M}_nB_\om$ (that is, it achieves the maximum number of points $2\epsilon_n$-apart in $\norm{\cdot}_X$ that can fit in $\overline{M}_nB_\om$), then the balls $h_j+2\epsilon_nB_X$ cover $\overline{M}_nB_\om$
and combining with \eqref{eq:bigball} we get that \begin{equation}\label{eq:t:main_aux4}X_n\subset  \bigcup_{j=1}^N (h_j+4\epsilon_nB_X).\end{equation}

Combining \eqref{pf:adh} together with \eqref{eq:t:main_aux4} we obtain
 \[N(4\epsilon_n,X_n,\norm{\cdot}_X)
 \leq N \leq \exp\left(\frac{\overline{M}_n^p}p + \varphi_0(\epsilon_n)\right).\] 
Using the definitions of $M_n,\overline{M}_n$ we get
\begin{align*}
&\log N(4\epsilon_n,X_n,\norm{\cdot}_X)\\
&\leq \frac{2^p}p\left(K(C+1)n\epsilon_n^2\vee(1+\frac1n)^pf(\sqrt{K(C+1)}n^\frac12\epsilon_n)g(\epsilon_n)^{1-\frac{p}2}\right)+n\epsilon_n^2.\end{align*}
Finally, using that $n\epsilon_n^2\gtrsim 1$ and $f$ is non-decreasing with $f(a)\to\infty$ at most polynomially as $a\to\infty$, we {get \eqref{eq:adh22}} . This completes the proof.
\end{proof}

\section{Proofs of Section \ref{sec:l2}}
\begin{proof}[Proof of Lemma \ref{lem:bessup}]
{The proof is very similar to the proof of \cite[Theorem 5]{DS15}, taking into account {Proposition \ref{prop:01}}. We include it for the reader's convenience.}

We first show that for $s<\alpha$, it holds $\pe(B^s_q)=1$. Indeed, we have 
\[\E_\pe\norm{u}^q_{B^s_q}=\E\sum_{\ell=1}^\infty \ell^{q(\frac{s}d+\frac12)-1}|\gamma_\ell|^q|\xi_\ell|^q=\E|\xi_1|^q\sum_{\ell=1}^\infty\ell^{\frac{q(s-\alpha)}d-1}.\] If $s<\alpha$, then the expectation is finite, hence $\mu(B^s_q)=1$.

We next show that if $\pe(B^s_q)=1$ then $s<\alpha$. If $\mu(B^s_q)=1$ then $\norm{u}_{B^s_q}<\infty$ almost surely with respect to $\pe$, hence 
\begin{equation}\label{eq:bessup1}\sum_{\ell=1}^\infty \ell^{\frac{q(s-\alpha)}d-1}|\xi_\ell|^q<\infty, \quad\text{almost surely}.\end{equation}
By contradiction to the Law of Large Numbers, we then get that 
\begin{equation}\label{eq:bessup2}s<\alpha+\frac{d}q.\end{equation} Define $\zeta_\ell=\ell^{\frac{q(s-\alpha)}d-1}|\xi_\ell|^q$, which are independent non-negative random variables. By \cite[Proposition 4.14]{OV06}, \eqref{eq:bessup1} also implies that  \begin{equation}\label{eq:bessup3}\sum_{\ell=1}^\infty \E[\zeta_\ell \wedge 1]<\infty.\end{equation}
We have 
\begin{align*}
\E \zeta_\ell=\E[\zeta_\ell\mathbbm{1}_{\zeta_\ell\leq1}]+\E[\zeta_\ell\mathbbm{1}_{\zeta_\ell>1}]
&\leq \E[\zeta_\ell\wedge 1]+I_\ell,
\end{align*}
where 
\begin{equation*}
I_\ell=\E \left[\ell^{\frac{q(s-\alpha)}d-1}|\xi_\ell|^q\mathbbm{1}_{\{|\xi_\ell|>\ell^{\frac{\alpha-s}d+\frac1q}\}}\right]=c_p\ell^{\frac{q(s-\alpha)}d-1}\int_{\ell^{\frac{\alpha-s}d+\frac1q}}^\infty x^qe^{-\frac{x^p}p}dx,
\end{equation*}
where $c_p$ depends only on $p$.
Since $p,q\geq1$, it holds that $x^qe^{-\frac{x^p}p}\leq C_1e^{-C_2x},$ for constants $C_1, C_2>0$ sufficiently large and small, respectively. This results in the bound \[I_\ell\leq c_p\frac{ C_1}{C_2}\ell^{\frac{q(s-\alpha)}d-1}\exp(-C_2\ell^{\frac{\alpha-s}d+\frac1q}):=\iota_\ell,\] where by \eqref{eq:bessup2} $\iota_\ell$ are summable. Combining, we get that
\[\sum_{\ell=1}^\infty \E[\zeta_\ell]\leq \sum_{\ell=1}^\infty\E[\zeta_\ell\wedge1]+\sum_{\ell=1}^\infty \iota_\ell<\infty.\] We thus have \[\sum_{\ell=1}\E[\zeta_\ell]=\E[|\xi_1|^q]\sum_{\ell=1}^\infty \ell^{\frac{q(s-\alpha}d-1}<\infty,\] and therefore $s<\alpha$.

Finally, {Proposition \ref{prop:01}} implies that $\pe(B^s_q)=0,$ for $s\geq \alpha$.
\end{proof}


\begin{proof}[Proof of Proposition \ref{prop:ubl2}]
We examine each case separately.
\begin{enumerate}
\item[i)]For $q\geq 2$,
by {Lemmas \ref{lem:l2sb} and \ref{lem:infl2}}, we have that as $\epsilon\to0$ the concentration function satisfies
\begin{equation}\label{eq:concl2}
	\varphi_{w_0}(\epsilon) \lesssim
	\begin{cases}
		  \epsilon^{\frac{\beta p-\alpha p-d}{\beta}} + \epsilon^{-\frac d\alpha}, & {\rm for }\; \beta < \alpha + \frac dp, \\
		 {(-\log \epsilon)^{\frac{q-p}q} }+ \epsilon^{-\frac d\alpha},
		& {\rm for }\; \beta = \alpha + \frac dp, \\
		1 + \epsilon^{-\frac d\alpha}, & {\rm for }\; \beta > \alpha + \frac dp. \\
	\end{cases}
\end{equation}
Determining which term dominates in each case, we find  
\begin{equation}\label{eq:concl22}
	\varphi_{w_0}(\epsilon) \lesssim
	\begin{cases}
		  \epsilon^{\frac{\beta p-\alpha p-d}{\beta}},  & {\rm for }\; \alpha \geq \beta,  \\
	          \epsilon^{-\frac d\alpha}, & {\rm for }\; \alpha < \beta . \\
	\end{cases}
\end{equation}
Computing the minimal solution $\epsilon_n$ such that  $\varphi_{w_0}(\epsilon_n)\leq n\epsilon_n^2,$  we arrive at $\epsilon_n\asymp r_n^{\alpha,\beta,p, q}$. 
\item[ii)]For $q<2$ and $p\leq q$,
by {Lemmas \ref{lem:l2sb} and \ref{lem:infl2}}, we have that as $\epsilon\to0$ the concentration function satisfies
\begin{equation}\label{eq:concl23}
	\varphi_{w_0}(\epsilon) \lesssim
	\begin{cases}
		  \epsilon^{\frac{2\beta pq-2\alpha pq-2qd}{2\beta q+qd-2d}} + \epsilon^{-\frac d\alpha}, & {\rm for }\; \beta < \alpha + \frac dp, \\
		 {(-\log \epsilon)^{\frac{q-p}q} }+ \epsilon^{-\frac d\alpha}
		& {\rm for }\; \beta = \alpha + \frac dp, \\
		1 + \epsilon^{-\frac d\alpha}, & {\rm for }\; \beta > \alpha + \frac dp. \\
	\end{cases}
\end{equation}
In this case, determining which of the two terms dominates for $\beta<\alpha+\frac{d}p$ is a bit more complicated and leads to a quadratic equation for the value of $\alpha$ balancing the two terms. This quadratic equation has two solutions, a negative one which is rejected since $\alpha>0$ and $\alpha=\frac{\beta p -d +a}{2p}$, where $a$ as in the statement of the lemma. Notice that the expression under the square root in $a$ is positive by the assumption on $\beta$. We thus have the following bound:
\begin{equation}\label{eq:concl24}
	\varphi_{w_0}(\epsilon) \lesssim
	\begin{cases}
		  \epsilon^{\frac{2\beta pq-2\alpha pq-2qd}{2\beta q+qd-2d}}, & {\rm for }\; \alpha\geq \frac{\beta p -d+a}{2p}, \\
		\epsilon^{-\frac d\alpha},
		& {\rm for }\; \alpha< \frac{\beta p -d+a}{2p}. \\	\end{cases}
\end{equation}
Computing the minimal solution $\epsilon_n$ such that  $\varphi_{w_0}(\epsilon_n)\leq n\epsilon_n^2,$  we again arrive at $\epsilon_n\asymp r_n^{\alpha,\beta,p, q}$. 
\item[iii)]For $q< 2$, thus $p>q$, the proof is similar to (ii) above, using the expressions corresponding to this case from {Lemma \ref{lem:infl2}.}
\end{enumerate}
Finally, notice that the constants in all the used upper bounds on the concentration function from {Lemmas \ref{lem:l2sb} and \ref{lem:infl2}}, depend on $w_0$ only through its $B^\beta_q$ norm, hence so do the constants in the rates derived above.
\end{proof}



\begin{proof}[Proof of Proposition \ref{prop:rescaledrate}]
In all of the studied cases except when $\alpha=\beta-\frac{d}p$ for $q>p$, {Lemmas \ref{lem:l2sb} and \ref{lem:infl2}} result in a bound of the form
\begin{equation*}
\varphil_{w_0}(\epsilon)\lesssim \lambda^{-p}\epsilon^{-s}+(\epsilon/\lambda)^{-\frac{d}\alpha},
\end{equation*}
for some $s=s(\alpha, \beta, p, q)\geq0$. We optimize the choice of $\lambda$ by balancing the two terms, obtaining the choice
\begin{equation}\label{eq:lamo}\lambda\asymp\epsilon^{\frac{d-\alpha s}{d+\alpha p}}.\end{equation}
We then equate the resulting bound on $\varphil_{w_0}(\epsilon)$ with $n\epsilon^2$, thus getting 
 \begin{equation}\label{eq:epsdef}\epsilon_n\asymp n^{-(2+\frac{dp+ds}{d+\alpha p})^{-1}}.\end{equation}
The last rate coincides with the rate $m_n=n^{-\frac{\beta}{d+2\beta}}$, if and only if \[s=s_0(\alpha,\beta,p):=\frac{d+\alpha p -\beta p}{\beta}.\] The corresponding optimal $\lambda$ is 
 \begin{equation}\label{eq:lam}\lambda_n\asymp n^{\frac{\alpha-\beta}{d+2\beta}}.\end{equation}For $s>s_0,$ the rate $\epsilon_n$ is polynomially slower than $m_n$, and this holds for the optimal and hence any choice of $\lambda_n$. {The case $s<s_0$ does not arise, as seen below (it cannot arise as it would lead to a faster rate than the minimax rate $m_n$).} 
The proof thus proceeds by comparing the values of $s$ obtained in {Lemma \ref{lem:infl2}} to $s_0$.

Recall that $q<2$.
If $q\leq p$ and $\alpha\geq \beta-\frac{d}q$, we have 
\begin{equation}\label{eq:s}s=\frac{2\alpha pq+2pd-2\beta p q}{2\beta q+dq-2d}.\end{equation} It holds 
\begin{equation*}
	s-s_0=\frac{dq(\beta-\alpha-d/q)(p-2){+}(d^2+2\alpha d)(p-q)}{\beta(2\beta q+d q-2d)}\geq0,
\end{equation*}
since both numerator and denominator are positive by our assumptions.
The difference $s-s_0$ vanishes if and only if $\alpha=\beta-\frac{d}q$ and $q=p$, in which case $\epsilon_n\asymp m_n$ for $\lambda_n \asymp n^{-\frac{d}{p(d+2\beta)}}.$
In all other cases, for any $\lambda_n$, $\epsilon_n$ is polynomially slower than $m_n$.

If $\alpha<\beta-\frac{d}p$ for any relationship between $q$ and $p$, we have $s=0>s_0$, hence for any $\lambda_n$, $\epsilon_n$ is polynomially slower than $m_n$. 

If $q>p$ and $\alpha>\beta-\frac{d}p$, then $$s=\frac{2\alpha p q+ 2qd-2\beta p q}{2\beta q+q d-2d}.$$ It holds $$s-s_0=\frac{(\beta-\alpha-\frac{d}p)(q-2)dp}{\beta(2\beta q+qd-2d)}>0,$$ 
hence for any $\lambda_n$, $\epsilon_n$ is polynomially slower than $m_n$. 

We next turn to the case $q>p$ for $\alpha=\beta-\frac{d}p$ (recall $q<2$). 
By {Lemmas \ref{lem:l2sb} and \ref{lem:infl2}} we have the bound
\begin{equation*}
\varphil_{w_0}(\epsilon)\lesssim \lambda^{-p}(\log1/\epsilon)^{\frac{q-p}q}+(\epsilon/\lambda)^{-\frac{d}\alpha}.
\end{equation*}
We again optimize the choice of $\lambda$ by balancing the two terms, to find 
\begin{equation}\label{eq:lamlog}\lambda\asymp \epsilon^{\frac{d}{\beta p}}(\log1/\epsilon)^{\frac{q-p}{q}\frac{\beta p -d}{\beta p^2}}.\end{equation}
The resulting bound on $\varphil_{w_0}(\epsilon)$ is \[\varphil_{w_0}(\epsilon)\lesssim (\log1/\epsilon)^{\frac{d(q-p)}{\beta qp}}\epsilon^{-\frac{d}\beta}.\]
We equate this bound with $n\epsilon^2$, obtaining the equation
\[\epsilon_n^{\frac{d+2\beta}{\beta}}\log^{-\frac{(q-p)d}{\beta qp}}(1/\epsilon_n)=n^{-1}.\]
 We then solve for $\epsilon$ using \cite[Lemma 3]{MP03}, included {below as Lemma \ref{lem:MP} for the reader's convenience}. This gives the claimed value of $\bar{r}_n$ and we can in turn compute the value of $\lambda_n$, by plugging $\epsilon_n\asymp \bar{r}_n$ in \eqref{eq:lamlog}.
It holds that $\omega>0$ (see the expression for $\lambda_n$ in the statement), since $\alpha=\beta-\frac{d}p>0$.

{Finally, we return to the case $q<p$ and determine the best achievable rate. If $\alpha\leq \beta-\frac{d}q$, then $s=0$ and by \eqref{eq:epsdef}, the rate is \[\epsilon_n\asymp n^{-(2+\frac{dp}{d+\alpha p})^{-1}},\] which is (uniquely) optimized when we choose $\alpha$ as large as possible, $\alpha=\beta-\frac{d}q.$ Plugging this choice of $\alpha$ into the last expression and \eqref{eq:lamo}, we obtain the claimed values for $\bar{r}_n$ and $\lambda_n$, respectively. It remains to verify that for $\alpha>\beta-\frac{d}q$ the resulting rate  $\epsilon_n$ in \eqref{eq:epsdef} is not better than the aforementioned $\bar{r}_n$. Indeed, comparing $\epsilon_n$ and $\bar{r}_n$, we conclude that this is the case if and only if the following inequality holds \[s\geq p^2\frac{\alpha q+d-\beta q}{dq-dp+p\beta q},\] with equality if and only if the two rates coincide. Using the expression for $s$ from \eqref{eq:s}, we get that the last inequality is equivalent to 
\[\frac{2}{2\beta q+d q-2d}\geq \frac{p}{\beta p q+dq-dp},\] where notice that the two denominators are positive since $\beta\geq \frac{d}q-\frac{d}2$ and $p\leq2$. It is then straightforward to see that the last inequality is equivalent to $p\leq2$ and so indeed the rate for $\alpha>\beta-\frac{d}q$ is strictly worse than $\bar{r}_n$ when $p<2$, while for $p=2$ the rates are identical for any $\alpha\geq\beta-\frac{d}q$.}

As always, notice that the constants in all the used upper bounds on the concentration function from {Lemmas \ref{lem:l2sb} and \ref{lem:infl2}}, depend on $w_0$ only through its $B^\beta_q$ norm, hence so do the constants in the rates derived above.
\end{proof}



\begin{proof}[Proof of Lemma \ref{lem:infl2}]
Let $w_0:=(w_{0,\ell})_{\ell\in\N}\in B^\beta_q$, so that under the assumption on $\beta$ it holds that $w_0\in\ell_2$, {see Lemma \ref{lem:emb} below}. Since $w_0\in B^\beta_q$ we find that
\begin{equation}
	\label{eq:asymp_of_w_ell}
	|w_{0,\ell}| \leq\norm{w_0}_{B^\beta_q} \ell^{-\frac\beta d-\frac 12+\frac1q}.
\end{equation}

Consider now approximations $h_{1:L} = (w_1, ..., w_L, 0, ...) \in \R^\infty$ of $w_0$, where $L\in \N$. Obviously, we have $h_{1:L} \in \oma$ for any $L$. We first study how large $L$ needs to be, in order to have
\begin{equation}
	\label{eq:h_L_condition}
	\norm{h_{1:L} - w_0}_{\ell_2} \leq \epsilon.
\end{equation}
In case $q>2$, we have
\begin{align*}
\norm{h_{1:L}-{w_0}}^2_{\ell_2}&=\sum_{\ell>L}\ell^{-\frac{2\beta}d-1+\frac2q} \ell^{\frac{2\beta}d+1-\frac2q}w_{0,\ell}^2 \\
&\leq\left(\sum_{\ell>L}\ell^{\frac{\beta q}d+\frac{q}2-1}|w_{0,\ell}|^q\right)^\frac2q
 \left(\sum_{\ell>L}\ell^{-\frac{2\beta q}{d(q-2)}-1}\right)^{\frac{q-2}q}
\leq \norm{w_0}_{B^\beta_q}^2 L^{-{\frac{2\beta}d}},
\end{align*}
where we have used the H\"older inequality $(\frac{q}2, \frac{q}{q-2})$ and comparison of the sum to an integral. 
If $q=2$, we obtain
\begin{equation*}
	\norm{h_{1:L}-{w_0}}^2_{\ell_2}=\sum_{\ell>L}\ell^{\frac{2\beta}d}w_{0,\ell}^2 \ell^{-\frac{2\beta}d}
	\leq  L^{-\frac{2\beta}d} \norm{w_0}_{B^\beta_q}^2.
\end{equation*}
If $q<2$, applying \eqref{eq:asymp_of_w_ell} yields
\begin{align*}
	\norm{h_{1:L}-{w_0}}^2_{\ell_2}&=\sum_{\ell>L}|w_{0,\ell}|^q|w_{0,\ell}|^{2-q}\leq\norm{w_0}^{2-q}_{B^\beta_q} \sum_{\ell>L}|w_{0,\ell}|^q\ell^{\frac{\beta q}d+\frac{q}2-1}\ell^{-\frac{2\beta}d-1+\frac{2}q}\\&\leq  L^{\frac2q-1-\frac{2\beta}d}\norm{w_0}^2_{B^\beta_q},
\end{align*} where we used the assumption on $\beta$ and where all the constants depend on $w_0$ only through its $B^\beta_q$-norm.

For \eqref{eq:h_L_condition} to hold with minimal $L\in\N$, we choose $L$ as 
\begin{equation}\label{eq:Lchoice}L=c 
	\begin{cases}	
		\epsilon^{-\frac{d}\beta}, & {\rm if }\; q\geq2, \; \\		
		\epsilon^{\frac1{\frac1q-\frac12-\frac{\beta}d}},  & {\rm if }\; q<2, 
\end{cases}\end{equation} where the constant $c$ depends on $w_0$ only through its $B^\beta_q$-norm.

Testing $h_{1:L} \in \oma$  in the infimum we aim to bound, yields
an upper bound 
\begin{equation*}
	I(\epsilon) := \inf_{h\in\oma:\norm{h-{w_0}}_{\ell_2}\leq\epsilon}\norm{h}_{\oma}^p
	\leq \norm{h_{1:L}}_{\oma}^p
	=\sum_{\ell\leq L} \ell^{\frac{p}2+\frac{\alpha p}d}  |w_{0,\ell}|^p.
\end{equation*} 

\begin{itemize}
\item[i)] For $q\leq p$, we use \eqref{eq:asymp_of_w_ell} to get
\begin{align}\label{eq:qleqp} \nonumber
\norm{h_{1:L}}^p_{\oma}&=\sum_{\ell\leq L}\ell^{\frac{p}2+\frac{\alpha p}d}|w_{0,\ell}|^p=\sum_{\ell\leq L}
\ell^{\frac{\beta q}d+\frac{q}2-1}|w_{0,\ell}|^q|w_{0,\ell}|^{p-q}\ell^{\frac{p}2+\frac{\alpha p}d -\frac{\beta q}d-\frac{q}2+1}\\\nonumber
&\leq\norm{w_0}_{B^\beta_q}^{p-q}\sum_{\ell\leq L}\ell^{\frac{\beta q}d+\frac{q}2-1}|w_{0,\ell}|^q\ell^{\frac{(\alpha-\beta)p}d+\frac{p}q}\\
&  \le  
	\begin{cases}	
		\norm{w_0}_{B^\beta_q}^p, & {\rm if }\; \beta \geq\alpha + \frac dq, \; \\
	        \norm{w_0}_{B^\beta_q}^pL^{\frac{(\alpha-\beta)p}d+\frac{p}q},& {\rm if }\; \beta <\alpha + \frac dq. \; \\
\end{cases}
\end{align}

\item[ii)] For $q> p$, we use H\"older inequality with $(\frac{q}p,\frac{q}{q-p})$ to bound
\begin{align}\label{eq:qlargerp} \nonumber
\norm{h_{1:L}}^p_{\oma}&=\sum_{\ell\leq L}\ell^{\frac{p}2+\frac{\alpha p}d}|w_{0,\ell}|^p=\sum_{\ell\leq L}
\ell^{\frac{\beta p}d+\frac{p}2-\frac{p}q}|w_{0,\ell}|^p\ell^{\frac{\alpha p}d -\frac{\beta p}d+\frac{p}q}\\ \nonumber
&\leq\left(\sum_{\ell\leq L}\ell^{\frac{\beta q}d+\frac{q}2-1}|w_{0,\ell}|^q\right)^{\frac{p}q}\left(\sum_{\ell\leq L}\ell^{(\frac{\alpha p}d -\frac{\beta p}d+\frac{p}q)\frac{q}{q-p}}\right)^{\frac{q-p}q}\\
&  \leq  
	\begin{cases}	
		\norm{w_0}_{B^\beta_q}^p, & {\rm if }\; \beta \geq\alpha + \frac dp, \; \\
		\norm{w_0}_{B^\beta_q}^p(\log L)^{\frac{q-p}q}, & {\rm if }\; \beta =\alpha + \frac dp, \; \\
	        \norm{w_0}_{B^\beta_q}^pL^{\frac{\alpha p-\beta p+d}d},& {\rm if }\; \beta <\alpha + \frac dp. \; \\
\end{cases}
\end{align}

\end{itemize}
Combining the bounds \eqref{eq:qleqp} and \eqref{eq:qlargerp} with the choice of $L$ from \eqref{eq:Lchoice} completes the proof.
 \end{proof}



\begin{proof}[Proof of Lemma \ref{lem:l2dom}]
We first show that $f$ and $g$ {defined by \eqref{eq:def_f_and_g}} 
satisfy {\eqref{eq:inf}.} 
Let $h=(h_\ell)\in aB_{\sha}$ and define
$x_{1:L}=(h_1, ..., h_L, 0, ...)$ for the smallest $L = L(\epsilon;a) \in \N$ such that
$\gamma_L \leq \frac{\epsilon}{a}.$ Such $L$ satisfies
$$L=\lceil a^{\frac {2d}{d+2\alpha}} \epsilon^{-\frac {2d}{d+2\alpha}}\rceil\leq  a^{\frac {2d}{d+2\alpha}} \epsilon^{-\frac {2d}{d+2\alpha}}+1.$$
Clearly, $x_{1:L}\in \oma$ and we obtain 
\begin{equation*}
\norm{h-x_{1:L}}_{\ell_2}^2 \leq \gamma_L^2 \sum_{\ell >L}\gamma_\ell^{-2} h_\ell^2 \leq \gamma_L^2 \norm{h}^2_{\sha}
 \leq \epsilon^2.
\end{equation*}
Consequently, by applying the H\"older inequality for $\left(\frac2p, \frac2{2-p}\right)$ we obtain
\begin{eqnarray*}
\inf_{x\in \oma:\norm{h-x}_{\ell^2}\leq \epsilon}\norm{x}_{\oma}^p
& \leq & \sum_{\ell\leq L}\gamma_\ell^{-p} |h_\ell|^p \\
& \leq & L^{1-\frac{p}2}\left(\sum_{\ell \leq L} \gamma_\ell^{-2} h_\ell^2\right)^{\frac{p}2} \\
& \leq & (a^{\frac {2d}{d+2\alpha}} \epsilon^{-\frac {2d}{d+2\alpha}}+1)^{1-\frac{p}2}a^p \\
& \leq & 2^{1-\frac{p}2}(a^{\frac {2d}{d+2\alpha}} \epsilon^{-\frac {2d}{d+2\alpha}}\vee1)^{1-\frac{p}2}a^p \\
& \leq & 2^{1-\frac{p}2}(a^\frac{2d-pd}{d+2\alpha}\vee1)(\epsilon^{-\frac{2d-pd}{d+2\alpha}}\vee1)a^p\\
& = & f(a)g(\epsilon)^{1-\frac{p}2}.
\end{eqnarray*}

For the second part of the claim, let $\epsilon_n=r_n^{\alpha,\beta,p,q}$ as in {Proposition \ref{prop:ubl2}} 
and observe that since $n\epsilon_n^2>1$ and $p\leq2$ we have $$f(n^\frac12\epsilon_n)=n^\frac{d+\alpha p}{d+2\alpha}\epsilon_n^{\frac{2d+2\alpha p}{d+2\alpha}}$$ and since $\epsilon_n<1$ we have $g(\epsilon_n)=2\epsilon_n^{-\frac{2d}{d+2\alpha}}$. For $p=2$ the claim holds trivially since  $f(n^\frac12\epsilon_n)=n\epsilon_n^2$ and $g(\epsilon_n)^{1-\frac{p}2}=1$. For $p\in[1,2)$, first notice that \[f(n^\frac12\epsilon_n)g(\epsilon_n)^{1-\frac{p}2}=2^{1-\frac{p}2}\epsilon_n^pn^{\frac{d+p\alpha}{d+2\alpha}},\] so that $f(n^\frac12\epsilon_n)g(\epsilon_n)^{1-\frac{p}2}\lesssim n\epsilon_n^2$ is equivalent to \begin{equation}\label{eq:optbnd}\epsilon_n\gtrsim n^{-\frac{\alpha}{d+2\alpha}}.\end{equation}
This is always true, since in the proof of {Proposition \ref{prop:ubl2}} 
we computed $\epsilon_n$ by finding an upper bound, say $B(\epsilon_n)$, on the concentration  function (see for example \eqref{eq:concl2}), and then solving $B(\epsilon_n)=n\epsilon_n^2$. Since $B(\epsilon_n)\geq \epsilon^{-\frac{d}\alpha}$, we have that \eqref{eq:optbnd} holds. 
\end{proof}



\begin{proof}[Proof of Lemma \ref{lem:l2domresc}]
Fix $a, \epsilon>0$ and let $h\in aB_{\shl}$. We need to show that $\f, \g$ satisfy the bound
\begin{equation}\label{eq:cmpl}\inf_{x\in\oml:\norm{h-x}_{\ell_2}\leq \epsilon}\norm{x}_{\oml}^p\leq \f(a)\g(\epsilon)^{1-\frac{p}2}.\end{equation}
Noticing that $a B_{\shl}=\lambda a B_{\sha}$ and using {Lemma \ref{lem:l2dom}}, we get that
\[\inf_{x\in \oml:\norm{h-x}_{\ell_2}\leq \epsilon }\norm{x}^p_{\oml}=\inf_{x\in\om:\norm{h-x}_{\ell_2}\leq\epsilon}\lambda^{-p}\norm{x}_{\oma}^p\leq \lambda^{-p}f(\lambda a)g(\epsilon)^{1-\frac{p}2}=\f(a)\g(\epsilon)^{1-\frac{p}2},\] hence \eqref{eq:cmpl} is indeed satisfied.

For the second part of the claim, let $\epsilon_n=\bar{r}_n$, for $\bar{r}_n$ as in {Proposition \ref{prop:rescaledrate}}
and observe that since $n\epsilon_n^2>1$ and $p\leq2$ we have $\f(n^\frac12\epsilon_n)=\lambda_n^{\frac{(2-p)d}{d+2\alpha}}n^\frac{d+\alpha p}{d+2\alpha}\epsilon_n^{\frac{2d+2\alpha p}{d+2\alpha}}$ and since $\epsilon_n<1$ we have $g(\epsilon_n)=2\epsilon_n^{-\frac{2d}{d+2\alpha}}$. 

For $p=2$ the claim holds trivially since  $\f(n^\frac12\epsilon_n)=n\epsilon_n^2$ and $\g(\epsilon_n)^{1-\frac{p}2}=1$. For $p\in[1,2)$, first notice that \[\f(n^\frac12\epsilon_n)\g(\epsilon_n)^{1-\frac{p}2}=2^{1-\frac{p}2}\lambda_n^{\frac{(2-p)d}{d+2\alpha}}\epsilon_n^pn^{\frac{d+p\alpha}{d+2\alpha}},\] so that $\f(n^\frac12\epsilon_n)\g(\epsilon_n)^{1-\frac{p}2}\lesssim n\epsilon_n^2$ is equivalent to \begin{equation}\label{eq:optbndl}\epsilon_n\gtrsim n^{-\frac{\alpha}{d+2\alpha}}\lambda_n^{\frac{d}{d+2\alpha}}.\end{equation}
This is always true, since in the proof of {Proposition \ref{prop:rescaledrate}} we computed $\epsilon_n$ by finding an upper bound, say $B(\epsilon_n,\lambda_n)$, on the concentration  function, and then solving $B(\epsilon_n,\lambda_n)=n\epsilon_n^2$. Since $B(\epsilon_n,\lambda_n)\geq \epsilon_n^{-\frac{d}\alpha}\lambda_n^{\frac{d}\alpha}$, we have that \eqref{eq:optbndl} holds. 
\end{proof}


\section{Proofs of Section \ref{sec:l8}}
\begin{proof}[Proof of Proposition \ref{prop:hol}] 
The proof follows the techniques of the proof of  \cite[Corollary 5]{DS15} taking into account the form of the Schauder basis functions $\psi_{kl}$.  In particular recall that $\psi_{kl}$ are $\vartheta$-H\"older continuous with $\vartheta\le S\wedge 1$, $S>\alpha$. Denote by $\kappa_n(W)$ the $n$th cumulant of a random variable $W$. Let $u\sim\pe$. Since the odd cumulants of centered random variables are zero and the cumulants are additive for independent random variables, we have for any integer $q\geq1$ and any $x,y\in[0,1], x\neq y$
\begin{align*}
&|\kappa_{2q}(u(x)-u(y))|  =  \left|\sum_{k=1}^\infty\sum_{l=1}^{2^k}\kappa_{2q}(\xi_{kl})2^{-(1+2\alpha)qk}(\psi_{kl}(x)-\psi_{kl}(y))^{2q}\right|\\
&\leq  C_q\sum_{k=1}^\infty\sum_{l=1}^{2^k}2^{-(1+2\alpha)qk}|\psi_{kl}(x)-\psi_{kl}(y)|^{2q} \\
&{\leq C_q\sum_{k=1}^\infty\min\Big\{  \sum_{l=1}^{2^k}2^{2(\vartheta-\alpha)qk}|x-y|^{2q \vartheta}, 2^{-(1+2\alpha) q k}\sum_{l=1}^{2^k}(|\psi_{kl}(x)|^{2q}+|\psi_{kl}(y)|^{2q})   \Big\}}\\
&{\leq C_q\sum_{k=1}^\infty \min\Big\{  2^k 2^{2(\vartheta-\alpha)qk}|x-y|^{2q \vartheta}, 2^{-2\alpha q k}   \Big\}}\\
&{\leq   C_q\sum_{\ell=1}^\infty \min\Big\{\ell^{2(\vartheta-\alpha)q}|x-y|^{2q \vartheta},\ell^{-2\alpha q-1}\Big\}}\\
 &{\leq  C_q\left(|x-y|^{2q \vartheta}\int_{1}^{|x-y|^{\frac{-2q\vartheta}{1+2q\vartheta}}}\ell^{2(\vartheta-\alpha)q}d\ell+\int_{|x-y|^{\frac{-2q\vartheta}{1+2q\vartheta}}}^\infty \ell^{-2\alpha q-1}d\ell\right)} \\
&{ \leq  C_q |x-y|^{\frac{4\alpha \vartheta q^2}{1+2q\vartheta}}},
\end{align*} 
where $C_q$ is a constant depending only on ($q, p$),  changing from line to line. {For 
the second inequality we used the $\vartheta$-H\"older continuity of $\psi_{kl}$ as described in (21) 
and the bound $(a+b)^{2q}\leq 2^{2q-1}(a^{2q}+b^{2q}),$ for any $a, b\in\R$. For the third inequality we used (22), which in turn implies that 
\[\norm{\sum_{l=1}^{2^k}|\psi_{kl}|}_{L_\infty}\leq C_22^{\frac{k}2},\] by letting $x^\ast\in \argmax_{x}\sum_{l=1}^{2^k}|\psi_{kl}(x)|$ and taking $u_{kl}=\sign\{\psi_{kl}(x^\ast)\}$.}

 Since the random variables $u(x)$ are centered, all moments of even order $2q, q\geq1$, can be written as homogeneous polynomials of the even cumulants of order up to $2q$, hence \[\E|u(x)-u(y)|^{2q}\leq C_q|x-y|^{\frac{4\alpha \vartheta q^2}{1+2q\vartheta}},\]
uniformly for $x,y\in[0,1].$ The result follows from Kolmogorov's continuity theorem, since we can choose $q$ arbitrarily large, see  {\cite[Corollary 4]{DS15}}.
\end{proof}

\begin{proof}[Proof of Proposition \ref{prop:supnorm}]
As a first step we generalize \cite[Lemma 2.1]{WS96} which holds for standard jointly normal variables, to the case of independent $p$-exponential variables with $p\in[1,2]$. Due to independence, we can use the product rule instead of Sidak's inequality. The lemma immediately generalizes due to the estimates in {Lemma \ref{lem:stolz} below}.

To get the result we then follow the proof of \cite[Theorem 1.3]{WS96}.  For $u$ drawn from an $\alpha$-regular $p$-exponential measure, it holds by  \eqref{eq:supnorm} 
 \[\norm{u}_{L_\infty}\leq c\sum_{k=0}^\infty2^{\frac{k}2}\sup_{1\leq l\leq 2^k}|u_{kl}|\leq c \sum_{k=0}^\infty2^{-\alpha k}\sup_{1\leq l\leq 2^k}|\xi_{kl}|.\] For $\epsilon>0$, let $n$ be an integer such that \[\frac{2}{1-2^{-\alpha/2}}2^{-\alpha n}\leq \epsilon<\frac{2}{1-2^{-\alpha/2}}2^{-\alpha(n-1)}.\] Define \[b_k:=\left\{\begin{array}{ll}2^{\frac32(k-n)\alpha}, & 
\;\mbox{if  $k<n$} 
                                     \\ 2^{\frac12(k-n)\alpha}, & \;\mbox{if  $k\geq n$}
                                    \end{array}\right.\] and notice that \[\sum_{k=0}^\infty2^{-\alpha k}b_k\leq \frac{2}{1-2^{-\alpha/2}}2^{-\alpha n}.\] Then, if $|\xi_{kl}|\leq b_k$ for all $k\geq 0$, $1\leq l\leq 2^k$, we have that $\norm{u}_{L_\infty}\leq\epsilon$. Therefore, by \cite[Lemma 2.1]{WS96} we have \[\rp(\norm{u}_{L_\infty}\leq \epsilon)\geq \exp(-C2^n)\] and the proof is complete since $2^n$ is of order $\epsilon^{-\frac1\alpha}$. 
\end{proof}

\begin{proof}[Proof of Lemma \ref{lem:infl8}]
Let $w_{kl}$ be the coefficients of $w_0$ in the wavelet basis $\psi_{kl}$. 
Consider  $h_{1:K}\in C[0,1]$ with coefficients $h_{kl}=w_{kl}$ for $k\leq K$ and $1\leq l\leq 2^k$ and $h_{kl}=0$ for $k>K$. Then $h_{1:K}\in\oma$ for any $K\in\N$ and  \[\norm{h_{1:K}-{w_0}}_{L_\infty}\leq c\sum_{k>K}2^{\frac{k}2}\sup_{1\leq l \leq 2^k}|w_{kl}|\leq c\norm{w_0}_{B^\beta_{\infty\infty}}\sum_{k>K}2^{-\beta k}\leq c\norm{w_0}_{B^\beta_{\infty\infty}}2^{-\beta K},\] for $c>0$ a changing constant independent of $w_0$ and where for the first inequality we used \eqref{eq:supnorm}. 
Choosing $K\in\N$ minimal so that $\epsilon\geq c\norm{w_0}_{B^\beta_{\infty\infty}}2^{-\beta K}$, we get $\norm{h_{1:K}-w_0}_{L_\infty}<\epsilon$. Then 
\begin{align*}
&\inf_{h\in \oma:\norm{h-w_0}_{L_\infty}\leq\epsilon}\norm{h}_{\oma}^p\leq\norm{h_{1:K}}^p_{\oma} =\sum_{k\leq K}\sum_{1\leq l\leq 2^k}|w_{kl}|^p2^{p(\frac12+\alpha)k} 2^{-\beta k}2^{\beta k}\\
&
\leq \norm{w_0}^p_{B^\beta_{\infty\infty}}
\sum_{k\leq K}\sum_{1\leq l\leq 2^k}2^{p(\alpha-\beta)k}=\norm{w_0}^p_{B^\beta_{\infty\infty}}\sum_{k\leq K}2^{(p(\alpha-\beta)+1)k}.
\end{align*}
The sum on the right hand side converges as $K\to\infty$ if $\beta>\alpha+\frac1p$, for $\beta= \alpha+\frac1p$ blows-up as $K\asymp\log(1/\epsilon)$ and for $\beta<\alpha+\frac1p$ blows-up as  $2^{K((\alpha-\beta)p+1)}\asymp\epsilon^{\frac{\beta p -\alpha p-1}\beta}$. 
\end{proof}

\begin{proof}[Proof of Lemma \ref{lem:l8dom}]
Let $h=(h_{kl})\in aB_{\sha}$ 
and define the approximation $x_{1:K}=\sum_{k\leq K}\sum_{l=1}^{2^k}h_{kl}\psi_{kl}$ for $K$ to be determined below. We have that $x_{1:K}\in{\oma}, \forall K\in \N$ and by \eqref{eq:supnorm} 
and Cauchy-Schwarz inequality, we get the bound
\begin{align*}&\norm{h-x_{1:K}}_{L_\infty}
\leq c\Bigg(\sum_{k=K+1}^\infty 2^{-2\alpha k}\Bigg)^\frac12\Bigg(\sum_{k=1}^\infty 2^{(1+2\alpha)k} \sup_{1\leq l\leq 2^k}|h_{kl}|^2\Bigg)^\frac12\leq c2^{-\alpha K}\norm{h}_{\sha},\end{align*} where $c>0$ is a constant with a value that changes below, dependent only on the Schauder basis, $\alpha$ and later on $p$.
For $K=K(\epsilon;a)\in\N$ minimal such that $\epsilon\geq ca 2^{-\alpha K}$, we have that $\norm{h-x_{1:K}}_{L_\infty}\leq\epsilon$. This $K$  satisfies $K=\lceil\log_2(ca^{\frac1\alpha}\epsilon^{-\frac{1}{\alpha}}) \rceil\leq \log_2(ca^{\frac1\alpha}\epsilon^{-\frac{1}{\alpha}})+1$, hence $2^K\leq 2ca^\frac1\alpha\epsilon^{-\frac1\alpha}$ and by H\"older inequality we have the bound 
\begin{align*}\inf_{x\in {\oma}:\norm{h-x}_{\infty}\leq\epsilon}\norm{x}^p_{\oma}&\leq \norm{x_{1:K}}^p_{\oma}= \sum_{k\leq K}\sum_{l=1}^{2^k}\frac{|h_{kl}|^p}{{\gamma_{kl}^p}}
\leq (2^K-1)^{1-\frac{p}2}\norm{h}_{\sha}^p\leq c\epsilon^{-\frac{2-p}{2\alpha}}a^{\frac{2-p+2\alpha p}{2\alpha}}.\end{align*}
Therefore,  $f(a)=ca^\frac{2-p+2\alpha p}{2\alpha}$ and $g(\epsilon)=\epsilon^{-\frac{1}{\alpha}}$ satisfy \eqref{eq:inf} 
and the proof is complete.
\end{proof}



\section{Shift spaces of scaled independent product measures}
\begin{proposition}\label{p:shiftgen}
Let $\nu$ be the law of the scaled sequence $(\gamma_\ell \xi_\ell)_{\ell\in\N}$, where $\xi_\ell$ independent and identically distributed univariate random variables and $\gamma=(\gamma_\ell)$ deterministic decaying sequence of positive numbers. Assume that the common distribution of $\xi_\ell$, has finite Fisher information and variance and has a density $\rho_\ell$ with respect to the Lebesgue measure which is everywhere positive and continuous. Then for any $h\in\R^\infty$ it holds that the translated measure $\nu_h$ and $\nu$ are either singular or equivalent. The shift space of the measure ${\nu}$ is 
\[\sh(\nu)=\{h\in \R^\infty: \sum_{\ell=1}^\infty h_\ell^2\gamma_\ell^{-2}<\infty\}.\]
Furthermore, letting $\rho_{\ell, h_\ell}=\rho_\ell(\cdot-h_\ell)$, we have for $h\in \sh(\nu)$, \[\frac{d\nu_h}{d\nu}(u)=\lim_{N\to\infty}\prod_{\ell=1}^N \frac{d\rho_{\ell,h_\ell}}{d\rho_\ell}(u_\ell) \quad\quad in \;\;{L^1(\R^\infty,\nu)}.\]
\end{proposition}
\begin{proof}
The positivity and continuity assumption on the density of $\xi_\ell$, secures that for each $\ell$ we have that $\rho_\ell$ and the translate $\rho_{\ell, h_\ell}=\rho_\ell(\cdot-h_\ell)$ are equivalent. Hence by the Kakutani Theorem \cite[Theorem 2.12.7]{Boga_gaussian} $\nu$ and $\nu_h$ are either singular or equivalent.

The rest of the proof relies on \cite[Section 1]{AK03} which builds on \cite{LAS65}. In these papers it is shown that if $Z=(Z_1, Z_2, \dots)$ is a sequence of independent of random variables with variance $0<\sigma_j^2<\infty$, then a sufficient condition for $Z$ and $Z+\alpha$, where $\alpha=(\alpha_1, \alpha_2, \dots)$, to be singular is that $\sum_{j=1}^\infty\alpha_j^2\sigma_j^{-2}=\infty$. If in addition the 
Fisher information $I_j$ of $Z_j$ is finite for all $j$, then a necessary condition for $Z$ and $Z+\alpha$,  to be mutually singular is that $\sum_{j=1}^\infty \alpha_j^2I_j=\infty.$

In our assumed setting, since the Fisher information of $\gamma_\ell \xi_\ell$ is $I_\ell=\gamma_\ell^{-2} I$ where $I$ is, the assumed to be finite, Fisher information of $\xi_\ell$, and  since $\var(\gamma_\ell \xi_\ell)=\gamma_\ell^2\var(\xi_\ell)$, we have that the necessary and sufficient condition for the singularity of $\nu$ with $\nu_h$ is $\sum_{\ell=1}^\infty h_\ell^2\gamma_\ell^{-2}=\infty.$ Since $\nu$ and $\nu_h$ are either singular or equivalent, the shift space is as claimed.

The Radon-Nikodym derivative follows again from Kakutani theorem in the form presented in \cite[Theorem 2.7]{GDP06}, noting that in \cite[Section 1]{AK03} it is shown that the Hellinger integral $H(\nu, \nu_h)$ is positive when $\sum_{\ell=1}^\infty h_\ell^2I_\ell<\infty$. 
\end{proof}


\section{Estimates for the univariate $p$-exponential distribution}
\begin{lemma}\label{lem:stolz}
Let $\xi\sim f_p(x)$, where $f_p(x)\propto\exp(-\frac{|x|^p}p), x\in\R$, $p\in[1,2]$. Then there exist constants $0<r_1<1$ and $r_2>0$ depending only on $p$, such that 
\[\rp(|\xi|\leq x)\geq \left\{\begin{array}{ll}r_1x, & {\rm if}
\;\mbox{$0\leq x\leq 1$} 
                                     \\ \exp(-r_2\exp(-\frac1p x^p)), & {\rm if} \;\mbox{$x> 1.$}\end{array}\right.\]
\end{lemma}
\begin{proof}
For $x\leq 1$ we have \[\rp(|\xi|\leq x)=2\int_0^x c_pe^{-\frac{t^p}p}dt\geq 2\int_0^x c_p e^{-1/p}dt\geq r_1 x,\]
where $r_1=2c_pe^{-1/p}=\frac{e^{-\frac1{p}}p^{-\frac1p}}{\Gamma(1+\frac1p)}<1$ for $p\in[1,2]$.

For $x>1$, let \[g_p(x)=\rp(|\xi|\leq x)-\exp(-r_2\exp(-\frac1p x^p))=\frac{\int_0^xe^{-\frac{t^p}p}dt}{\int_0^\infty e^{-\frac{t^p}p}dt}-\exp(-r_2\exp(-\frac1p x^p)),\]
for some $r_2>0$ large enough, so that $g_p(1)>0$. Such an $r_2$ exists since the first term in $g_p(1)$ is fixed and positive and the second term is decreasing to zero as $r_2$ grows. The derivative of $g_p(x)$ is 
\[\frac{d}{dx}g_p(x)= e^{-\frac{x^p}p}\left(\frac{1}{\int_0^\infty e^{-\frac{t^p}p}dt}-r_2 x^{p-1}\exp(-r_2\exp(-\frac1p x^p))\right).\] The term inside the parenthesis, as $x\geq1$ grows, starts from a possibly positive value and is monotonically decreasing, eventually becoming negative. This means that   the derivative $\frac{d}{dx}g_p(x)$, as $x\geq1$ grows starts from a possibly positive value and eventually becomes negative too, and thus has at most one root which corresponds to at most a unique critical point of $g_p(x), x\geq1$, which if exists is a maximum. Noting that $\lim_{x\to+\infty}g_p(x)=0$, and since $g_p(1)>0$, we get that $g_p(x)\geq0, \forall x>1$ and the proof is complete.
\end{proof}

\section{Other technical results}\label{sec:tech}

\begin{lemma}\label{lem:emb}
Let $q\geq1, d\in\N$ and $\beta>\frac{d}q-\frac{d}2$. Then $B^\beta_q\subset \ell_2$.
\end{lemma}
\begin{proof}
Let $w:=(w_\ell)_{\ell\in\N}\in B^\beta_q$. For $q=2$ the claim is trivially true. If $q>2$, then by H\"older inequality for $(\frac q2, \frac q{q-2})$, we have 
\[\norm{w}_{\ell_2}^2=\sum_{\ell=1}^\infty w_\ell^2\ell^{2(\frac{\beta}d+\frac12)-\frac2q}\ell^{\frac2q-2(\frac{\beta}d+\frac12)}\leq\norm{w}_{B^\beta_q}^2\left(\sum_{\ell=1}^\infty\ell^{\frac{2q}{q-2}(\frac1q-\frac{\beta}d-\frac12)}\right)^{\frac{q-2}q},\]
where the last sum is finite if and only if $\beta>\frac{d}q-\frac{d}2$. 
If $q<2$,
\[\norm{w}_{\ell_2}^2=\sum_{\ell=1}^\infty w_\ell^2=\sum_{\ell=1}^\infty|w_\ell|^q\ell^{\frac{\beta q}d+\frac{q}2-1}\ell^{-\frac{\beta q}d-\frac{q}2+1}|w_\ell|^{2-q}\leq c\norm{w}_{B^\beta_q}^q,\] provided $\ell^{-\frac{\beta q}d-\frac{q}2+1}|w_\ell|^{2-q}\leq c$. Using that $w\in B^\beta_q$, we have 
\[|w_\ell| \leq\norm{w}_{B^\beta_q} \ell^{-\frac\beta d-\frac 12+\frac1q}.\]
The last estimate gives
\[ \ell^{-\frac{\beta q}d-\frac{q}2+1}|w_\ell|^{2-q}\leq \norm{w}_{B^\beta_q}^{2-q}\ell^{-\frac{2\beta}d-1+\frac2q},\] which is bounded for $\beta\geq \frac{d}q-\frac{d}2.$ 
\end{proof}

\begin{lemma}[Lemma 3 \cite{MP03}]\label{lem:MP}
Given $a, b>0$ consider the functions
\[r_{a,b}(s):=s^a\log^{-b}(1/s), \;0<s<1,\]
and \[v_{a,b}(s):=s^{1/a}\log^{b/a}(1/s^{1/a}), \;0<s<1.\]
Then we have that \[\lim_{s\to0} \frac{r_{a,b}^{-1}(s)}{v_{a,b}(s)}=1.\]
\end{lemma}

\section{Minimax and linear minimax rates and their relationship to posterior contraction rates, in the white noise model}\label{sec:disc}
Consider the estimation of a function $w\in L_2[0,1]$ observed under scaled Gaussian white noise, as in Section \ref{ssec:wn}. Let $d$ be a metric on $L_2[0,1]$, and consider the minimax risk over a class $\mathcal{F}\subset L_2[0,1]$
\begin{equation}\label{eq:minexp}\inf_{\tilde{w}_n}\sup_{w\in\mathcal{F}}\E_{P^{(n)}_w}[d(\tilde{w}_n,w)],\end{equation} where the infimum is taken over  all estimators $\tilde{w}_n$ constructed using the sample path $X^{(n)}\sim P^{(n)}_w$. \emph{The minimax rate of estimation in $d$-risk over $\mathcal{F}$}, is the fastest rate of decay $r_n$ of the minimax risk in \eqref{eq:minexp}, as $n\to\infty$. See \cite[Definition 6.3.1]{GN16} for details.

We can also consider a more general minimax framework, and in particular can embed convergence in $d$-loss in probability in the minimax framework; see \cite[Remark (2) in Section 2.1]{AT09}. In this case, we study the minimax risk
\begin{equation}\label{eq:minprob}\inf_{\tilde{w}_n}\sup_{w\in\mathcal{F}}P^{(n)}_w(d(\tilde{w}_n,w)\geq r_n).\end{equation}
 The \emph{minimax rate of estimation in $d$-loss in probability over $\mathcal{F}$}, is the fastest rate $r_n$ for which the minimax risk in \eqref{eq:minprob} vanishes as $n\to\infty$.
 
In both \eqref{eq:minexp} and \eqref{eq:minprob}, we can restrict the infimum to \emph{linear} estimators, in which case we have the corresponding notions of \emph{linear minimax rates}. The minimax and linear minimax rates in $L_2$-risk (as in \eqref{eq:minexp}) under Besov smoothness for function estimation in the white noise model can be found in \cite[Theorem 1]{DJ98}; see also \eqref{eq:nlminimax} and \eqref{eq:lminimax} in Section \ref{sec:l2} of the present article.

The minimax rate in $L_2$-loss \emph{in probability} (as in \eqref{eq:minprob}) is a benchmark for rates of contraction in $L_2$-loss, because if the posterior contracts at a rate $\epsilon_n$ at a $w_0\in L_2[0,1]$, then the center of the smallest ball containing at least half the posterior mass, is an estimator converging at the same rate $\epsilon_n$ in $L_2$-distance, in $P^{(n)}_{w_0}$-probability; see for example \cite[Theorem 8.7]{GV17}. 
Furthermore, the typical approach for establishing minimax rates in $L_2$-risk (more generally $d$-risk, as in \eqref{eq:minexp}) and in particular lower bounds, is by establishing a lower bound in probability and using Markov's inequality to obtain a lower bound in expectation, see for example \cite[Section 6.3.1]{GN16}. Hence, it is implicit in the typical derivation of the minimax rate in $L_2$-risk, that the same rate is also the minimax rate in $L_2$-loss in probability. This is indeed the approach used when establishing the minimax rates in $L_2$-risk for the white noise model under Besov-type smoothness, see for example \cite[Section 6.3.3]{GN16}, and hence the minimax rate $m_n$  defined in \eqref{eq:nlminimax} \emph{is} a benchmark for rates of contraction in $L_2$-loss under Besov-type smoothness.

{We next turn to the question of whether Gaussian priors are fundamentally limited by the minimax rate in $L_2$-risk over linear estimators. In the white noise model, Gaussian priors are conjugate to the Gaussian likelihood, hence the posterior is also Gaussian. By Anderson's inequality, see for example {Proposition 2.4}, the posterior mean which is a linear estimator, coincides with the center of the smallest ball containing at least half the posterior mass. Following the train of thought of the previous paragraph, one would thus expect that the contraction rates in $L_2$-loss of Gaussian priors under Besov-type smoothness cannot be faster than the linear minimax rate $l_n$ in $L_2$-risk defined in \eqref{eq:lminimax}. However, linear minimax rates in $L_2$-risk over Besov-bodies are established by directly working with $L_2$-risk and not by establishing lower bounds in probability, see \cite[Section 6]{DJ98} and \cite{DLM90}. In other words, the linear minimax rates in $L_2$-risk (as in \eqref{eq:minexp}) do not necessarily coincide with the linear minimax rates in $L_2$-loss in probability (as in \eqref{eq:minprob}), over Besov-bodies. In order to establish that Gaussian priors are fundamentally limited by the linear minimax rate $l_n$ in $L_2$-risk, one needs either to establish that $l_n$ is also the linear minimax rate in $L_2$-loss in probability (that is, to establish the corresponding lower bound in probability), or to show that, in this Gaussian-conjugate setting, posterior contraction in $L_2$-loss at a rate $\epsilon_n$ implies convergence in expected $L_2$-distance of the posterior mean at the same rate. Both of these tasks appear to be non-trivial.} 

{An alternative, more modest, approach for establishing that posterior contraction rates in $L_2$-loss for Gaussian priors under Besov-type smoothness cannot be faster than the linear minimax rate in $L_2$-risk, is to study lower bounds on the contraction rate and establish that the upper bounds obtained in this article are sharp. Lower bounds on contraction rates in $L_2$-loss for Gaussian priors under \emph{Sobolev} smoothness, have been studied in \cite[Theorem 2]{IC08}. Once more, it is not immediately obvious how this result can be generalized to obtain good lower bounds under Besov-type smoothness, since its proof relies on the weighted $\ell_2$-type structure of Sobolev spaces. }

{Although beyond the scope of the present paper, we believe that the question of whether Gaussian priors are limited by linear minimax rates in $L_2$-risk is extremely interesting: an affirmative answer would rigorously show that when interested in reconstructing spatially inhomogeneous unknown functions (that is functions in $B^\beta_q$ for $q<2$), from the point of view of contraction rates in $L_2$-loss, it is better to use a Laplace prior rather than a Gaussian prior.}

\end{document}